\documentclass[10pt, reqno]{amsart}
\usepackage{amsmath}
\usepackage{amsthm, color}
\usepackage{amsfonts}
\usepackage{amssymb}
\usepackage{mathrsfs}
\usepackage[margin=.95in]{geometry}
\usepackage{float}
\usepackage{set space}
\usepackage[dvipsnames]{xcolor}
\usepackage{tikz}
\usetikzlibrary{matrix, arrows, quotes, calc, intersections, automata, positioning}
\usepackage{tikz-cd}
\usepackage[inline]{enumitem}
\usepackage[colorlinks]{hyperref}
\hypersetup{linkcolor=black,citecolor=blue,filecolor=black,urlcolor=black} 
\usepackage{pdflscape}
\usepackage{array}
\usepackage{adjustbox}
\usepackage{tabularx}
\usepackage{multirow}
\usepackage{multicol}
\usepackage{arydshln}
\usepackage{comment}
\usepackage{ifthen}
\usepackage{caption}
\usepackage{bbm}

\renewcommand{\a}{\mathfrak{a}}
\newcommand{\abs}[1]{\left\vert #1 \right\vert}

\newcommand{\p}{\mathfrak{p}}

\renewcommand{\bar}[1]{\overline{#1}}
\DeclareMathOperator{\BC}{BC}               

\newcommand{\CC}{\mathbb{C}}

\DeclareMathOperator{\Chow}{CH}             

\DeclareMathOperator{\cl}{cl}		
\DeclareMathOperator{\coker}{Coker}

\newcommand{\dd}{\partial}

\renewcommand{\emptyset}{\varnothing}
\renewcommand{\epsilon}{\varepsilon}

\newcommand{\exterior}{\bigwedge}	        

\newcommand{\GMA}[1]{\mathsf{B}_{#1}}

\DeclareMathOperator{\HF}{\mathrm{HF}}
\DeclareMathOperator{\HS}{H}

\newcommand{\Iff}{\Leftrightarrow}

\newcommand{\Implies}{\Rightarrow}

\DeclareMathOperator{\init}{in}				
\newcommand{\iso}{\cong}
\DeclareMathOperator{\Ker}{Ker}
\renewcommand{\ker}{\Ker}
\newcommand{\kk}{\mathbbm{k}}					        
\renewcommand{\L}{\mathcal{L}}

\newcommand{\m}{\mathfrak{m}}

\newcommand{\onto}{\twoheadrightarrow}
\DeclareMathOperator{\OrlikSolomon}{\mathsf{A}}           		
\newcommand{\OS}[1]{\OrlikSolomon_{#1}}
\DeclareMathOperator{\Poin}{P}				
\renewcommand{\P}{\Poin}
\renewcommand{\phi}{\varphi}

\DeclareMathOperator{\rank}{rank}
\DeclareMathOperator{\rk}{rk}

\renewcommand{\setminus}{\smallsetminus}
\DeclareMathOperator{\si}{si}           

\newcommand{\term}[1]{\textbf{\textsf{#1}}}
 			
\newcommand{\tensor}{\otimes}

\DeclareMathOperator{\Tor}{Tor}

\newcommand{\ZZ}{\mathbb{Z}}

\newcommand{\graph}[3][scale = 1.2]{
	\begin{tikzpicture}[#1]          
	\newcommand*\points{#2}     
	\newcommand*\edges{#3}          
	\newcommand*\scale{0.75}          
	\foreach \x/\y/\l/\v/\a in \points {
		\draw[fill = black!50] (\scale*\x,\scale*\y) circle [radius = 0.1] node[label = {[label distance = 0.05 cm]\a: $\v$}] (\l) {}; 
	}
	\foreach \x/\y in \edges { \draw (\x) -- (\y); }      
	\end{tikzpicture}
}

\newcommand{\edgelabeledgraph}[3][scale = 1.2]{
	\begin{tikzpicture}[#1]          
	\newcommand*\points{#2}     
	\newcommand*\edges{#3}          
	\newcommand*\scale{0.75}          
	\foreach \x/\y/\z/\w/\a in \points {
		\draw[fill = black!50] (\scale*\x,\scale*\y) circle [radius = 0.1] node[label = {[label distance = 0.05 cm]\a: $\w$}] (\z) {}; 
	}
	\foreach \x/\y/\p/\a/\l in \edges { \draw (\x) -- (\y) node [pos = \p, sloped, \a] {\textcolor{WildStrawberry}{\small $\l$}}; }      
	\end{tikzpicture}
}

\newcommand{\lattice}[3][]{
	\begin{tikzpicture}[#1]          
	\newcommand*\points{#2}     
	\newcommand*\edges{#3}          
	\newcommand*\scale{0.015}          
	\foreach \x/\y/\z/\w in \points {
		\node (\w) at (\x, \y) {$\z$}; 
	}
	\foreach \x/\y in \edges { \draw (\x) -- (\y); }      
	\end{tikzpicture}
}
	
\newtheorem{thm}{Theorem}[section]
\newtheorem{lemma}[thm]{Lemma}
\newtheorem{prop}[thm]{Proposition}
\newtheorem{cor}[thm]{Corollary}
\newtheorem{mainthm}{Theorem}

\newtheorem{prob}[thm]{Open Problem}

\theoremstyle{definition}
\newtheorem{defn}[thm]{Definition}
\newtheorem{example}[thm]{Example}

\numberwithin{equation}{section}
\numberwithin{figure}{section}

\title[Koszul Graded M\"obius algebras]{Koszul Graded M\"obius Algebras \\ and Strongly Chordal Graphs} 
\date{\today}
\author[A. LaClair]{Adam LaClair}
\address{Purdue University, Department of Mathematics, West Lafayette, IN, USA}
\email{alaclair@purdue.edu}
\author[M. Mastroeni]{Matthew Mastroeni}
\address{SUNY Polytechnic Institute, Utica, NY, USA}
\email{mastromn@sunypoly.edu}
\author[J. McCullough]{Jason McCullough}
\address{Iowa State University, Department of Mathematics, Ames, IA, USA}
\email{jmccullo@iastate.edu}
\author[I. Peeva]{Irena Peeva}
\address{Cornell University, Department of Mathematics, Ithaca, NY, USA}
\email{ivp1@cornell.edu}

\begin{document}



\begin{abstract} The graded M\"{o}bius algebra of a matroid is a  commutative graded algebra which encodes the combinatorics of the lattice of flats of the matroid.  As a special subalgebra of the augmented Chow ring of the matroid, it plays an important role in the recent proof of the Dowling-Wilson Top Heavy Conjecture. Recently, Mastroeni and McCullough proved that the Chow ring and the augmented Chow ring of a matroid are Koszul.  We study when graded M\"obius algebras are Koszul.
We characterize the Koszul graded M\"obius algebras  of cycle matroids of graphs in terms of properties of the graphs. 
Our results yield a new characterization of strongly chordal graphs via edge orderings.
\end{abstract}

\maketitle

\begin{spacing}{1.1}

\section{Introduction}

\noindent  Given a field $\kk$, a standard graded $\kk$-algebra $A = \bigoplus_{i \ge 0}A_i$ with homogeneous maximal ideal $\m = \bigoplus_{i > 0} A_i$ is called \term{Koszul} if $A/\m \cong \kk$ has a linear free resolution over $A$.  Koszul algebras were formally defined by Priddy \cite{Priddy} as a way of unifying free resolution constructions in topology and representation theory, and they have been studied for their extraordinary homological and duality properties ever since.  They appear in many areas of algebra, geometry, and topology, such as coordinate rings of canonical curves, coordinate rings of Grassmannians in their Pl\"ucker embeddings, and sufficiently high Veronese subalgebras of any standard graded algebra.  We refer the reader to 
\cite{Koszul:algebras:and:regularity}, 
 \cite{Koszul:algebras:and:their:syzygies}, \cite{Froberg:Koszul:algebras:survey}, and \cite{PP}  for expository overviews.  
 
 This paper studies under what conditions the graded M\"obius algebra of a matroid is Koszul.  
  Let $M$ be a simple matroid with finite ground set $E$ and lattice of flats $\L$.  The \term{graded M\"obius algebra} of $M$ is the commutative ring 
\[ \GMA{M} = \bigoplus_{F \in \L} \kk y_F \]
having the elements $y_F$ for each flat $F$ of $M$ as a $\kk$-basis, with multiplication defined by
\[ y_Fy_G = \left\{\begin{array}{cl} y_{F \vee G}, & \text{if}\; \rk(F \vee G) = \rk F + \rk G \\[1 ex] 0, & \text{otherwise.}\end{array}\right. \] 
 With this multiplication, $\GMA{M}$ is a standard graded algebra with grading induced by the rank function of the lattice. The algebra has also been denoted by $B^\ast(M)$ in \cite{HW17}, by $Q/J_M$ in \cite{Maeno:Numata:published}, and by $\mathrm{H}(M)$ in \cite{semi-small:decompositions}.  Note that this algebra is distinct from the (ungraded) M\"obius algebra of Greene \cite{Greene73} or Solomon \cite{Solomon67}, which omits the rank condition.

Graded M\"obius algebras were defined in \cite{singular:Hodge:theory:for:combinatorial:geometries} and  \cite{semi-small:decompositions} as an algebraic tool for resolving the longstanding  Dowling-Wilson Top Heavy Conjecture concerning the numbers of flats of a given rank in a matroid.  In the representable case, the graded M\"obius algebra  is isomorphic to the cohomology ring of the associated matroid Schubert variety studied by Ardila and Boocher \cite{Ardila-Boocher:Schubert:varieties}; see \cite[Theorem 14]{HW17} and \cite[Section 1.3]{singular:Hodge:theory:for:combinatorial:geometries} for precise statements. By construction, the Hilbert function of $\GMA{M}$ records exactly the number of flats of $M$ of a given rank, called the Whitney numbers of the second kind. As the graded M\"obius algebra of $M$ embeds as a subalgebra of its augmented Chow ring $\Chow(M)$ \cite[2.15]{semi-small:decompositions}, the proof of the Top-Heavy Conjecture comes down to interpolating between the combinatorics of the graded M\"obius algebra and the good algebraic properties of the augmented Chow ring:  Poincar\'e duality, the Hard Lefschetz Theorem, and the Hodge-Riemann relations.  

Dotsenko had conjectured \cite{Dotsenko} that the Chow ring of any matroid should   be a Koszul algebra.  It is well-known that every Koszul algebra has a presentation with a defining ideal generated by quadratic forms; such algebras are called \term{quadratic}.  However, not all quadratic algebras are Koszul (for example, see Section \ref{open:problems}).  By far the most common way of proving that a quotient of a polynomial ring or exterior algebra is Koszul is to show that its defining ideal has a quadratic Gr\"{o}bner basis, possibly after a suitable linear change of coordinates; such algebras are called \term{G-quadratic}.  This was the approach used by Dotsenko to prove that the cohomology rings of the moduli spaces of stable rational marked curves are Koszul.  
Both these rings and the Chow rings of matroids fit into a larger framework of Chow rings associated to atomic lattices and building sets studied by Feichtner and Yuzvinksy \cite{Chow:rings:of:atomic:lattices}, which was the basis for Dotsenko's conjecture.    
Coron \cite{Coron} subsequently generalized Dotsenko's result to so-called supersolvable built lattices, proving that the associated Chow rings have quadratic Gr\"obner bases.  

In some rare cases, having a quadratic Gr\"{o}bner basis characterizes the Koszul property; this is known to hold for canonical rings of curves \cite{Groebner:flags}, Hibi rings of posets \cite{HH87}, toric edge rings of bipartite graphs \cite{OH},   quadratic Gorenstein rings of regularity 2 \cite{Groebner:flags}, and Orlik-Solomon algebras of graphic matroids \cite{broken:circuit:complexes}.  However, if no quadratic Gr\"{o}bner basis can be found, then proving Koszulness is usually quite challenging, see  \cite{Ca} and \cite{CC}.  In particular, Mastroeni and McCullough \cite{Chow:rings:of:matroids:are:Koszul} proved Dotsenko's conjecture for both the augmented and un-augmented versions of the Chow ring of a matroid by constructing a special family of ideals known as a Koszul filtration.

In contrast with the situation for Chow rings, the algebraic properties of graded M\"obius algebras have been less studied.  Graded M\"obius algebras first appeared as objects of secondary interest in work of Maeno and Numata \cite{Maeno:Numata:published}, who studied the Sperner property of modular lattices via certain Artinian Gorenstein rings defined via Macaulay inverse systems.  A consequence of their work is that $\GMA{M}$ is Gorenstein if and only if the lattice of flats of $M$ is modular \cite[5.7]{Maeno:Numata:published}.   The question of the Koszul property was raised by Ferroni et.~al.~\cite{FMSV22}, who observed that, in contrast to Chow rings of matroids, not all graded M\"obius algebras are Koszul.  {\it Currently, it is an open problem when the graded M\"obius algebra of a matroid is quadratic or Koszul.}

We  review some relevant background in Section \ref{background}.
In Section \ref{quadracity:and:GB's}, we give an explicit presentation for the graded M\"obius algebra of a matroid as well as Gr\"obner bases for its defining ideal.  This presentation (Proposition~\ref{GMA:presentation}) bears a number of similarities with the presentation of the much better studied Orlik-Solomon algebra $\OS{M}$ of a matroid $M$. In Sections \ref{MAT-labelings} and \ref{proof:of:main:thm}, we specialize to the case of graphic matroids to obtain sharper results.  We prove that  the graded M\"obius algebra $\GMA{G}$ of the cycle matroid $M(G)$ of a graph $G$ is quadratic if and only if $G$ is chordal (Theorem~\ref{quadratic:GMA:iff:chordal}).  An analogous statement is known to hold for the Orlik-Solomon algebra $\OS{G}$ of $M(G)$, and is further equivalent to $\OS{G}$ being Koszul.  On the other hand, characterizing when the graded M\"obius algebra is Koszul is difficult even in the graphic case; one needs something stronger than chordality to ensure the Koszul property.  Our main theorem 
is:

\begin{mainthm}
Let $G$ be a graph with cycle matroid $M(G)$, and let $\GMA{G}$ denote the graded M\"obius algebra of $M(G)$.  The following are equivalent:
\begin{enumerate}[label = \textnormal{(\alph*)}]
    \item $M(G)$ is strongly T-chordal.
    \item $\GMA{G}$ has a quadratic Gr\"obner basis.
    \item $\GMA{G}$ is Koszul.
    \item $G$ is strongly chordal.
\end{enumerate}
\end{mainthm}

The definitions of T-chordal and strongly chordal are given in Sections~\ref{quadracity:and:GB's} and \ref{MAT-labelings}, respectively.
The implications (a) $\Implies$ (b) and (d) $\Implies$ (a) follow from Theorem \ref{GMA:quadratic:GB} and Corollary~\ref{broken:trampolines:are:G-quadratic}, respectively. It is well known that (b) $\Implies$ (c).  The implication (c) $\Implies$ (d) is established in
Theorem~\ref{Koszul:imples:strongly:chordal}.

We summarize in
Figure~\ref{fig1} and
Figure~\ref{graphic:quadracity:implications} the relationships between the quadracity and Koszul properties of Orlik-Solomon algebras and graded M\"obius algebras, and the various notions of chordality; the former figure covers the case of general matroids, and the latter is specific to graphic matroids.

As a byproduct of our work on Koszul graded M\"obius algebras, we obtain a new characterization of strongly chordal graphs in terms of edge orderings: 

\begin{mainthm}[Theorem \ref{strongly:chordal:implies:strongly:T-chordal}]
A graph $G$ is strongly chordal if and only if there is a total order $\prec$ on the edges of $G$ with the property that    for every cycle $C$ of size at least four in $G$ and every edge $e \in C \setminus \min_\prec C$, there is a chord $c$ of $C$ and edges $a, b \in C \setminus e$ such that $\{a, b, c\}$ is a 3-cycle with $c \succ \min(a, b)$.
\end{mainthm}
 
  At the end of the paper, we raise some open problems.

\section{Background} \label{background}

\noindent In this section, we review some relevant background material on matroids, Hilbert functions, and free resolutions.

\subsection{Matroids}\label{matroids}

A \term{matroid} is a pair $M = (E,\mathcal{I})$ consisting of a finite set $E$, called the \term{ground set} of $M$, and a collection $\mathcal{I}$ of subsets of $E$ satisfying three properties:
\begin{enumerate}
    \item $\varnothing \in \mathcal{I}$.
    \item If $I \in \mathcal{I}$ and $I' \subseteq I$, then $I' \in \mathcal{I}$.
    \item If $I_1,I_2 \in \mathcal{I}$ and $|I_1| < |I_2|$, then there exists an element $e \in I_2 \smallsetminus I_1$ such that $I_1 \cup e \in \mathcal{I}$.
\end{enumerate}
We will often write subsets $\{e_1, e_2, \dots, e_s\} \subseteq E$ as $e_1e_2\cdots e_s$ when there is no chance for confusion.  Thus, if $F \subseteq E$, we will frequently write $F \cup e$ and $F \setminus e$ in place of $F \cup \{e\}$ and $F \setminus \{e\}$ respectively.

The members of $\mathcal{I}$ are called \term{independent sets} of $M$.  A maximal independent set is called a \term{basis}.  A subset of $E$ that is not in $\mathcal{I}$ is called \term{dependent}.  A minimal dependent set of $M$ is called a \term{circuit}.   
All bases of a matroid have the same cardinality, called the \term{rank} of $M$.  Given a subset $X \subseteq E$, the \term{rank} of $X$, denoted $\rk_M X$, is the cardinality of the largest independent set contained in $X$; we drop the subscript when the matroid is clear from context.  The \term{closure} of a subset $X \subseteq E$ in $M$ is 
\[ \cl(X) = \{e \in E \mid \rk(X \cup e) = \rk X \}.\]  
A subset $F \subseteq E$ is called a \term{flat} of $M$ if $F = \cl(F)$. 
The set  of all flats of $M$ ordered by inclusion is a lattice denoted by $\L(M)$; for any two flats $F,G \in \L(M)$, the meet is the intersection, $F \wedge G = F \cap G$, and the join is the closure of the union, $F \vee G = \cl(F \cup G)$.  Matroids can also be characterized by their rank functions, by their bases, or by their circuits.

\begin{example} \label{lattice:of:flats}
    Every graph $G$ determines a cycle matroid $M = M(G)$ whose ground set is the set of edges of $G$ and whose independent sets are sets of edges forming acyclic subgraphs of $G$.  By construction, the circuits of $M$ are precisely the sets of edges forming minimal cycles in the graph.  Flats of $M$ correspond to subgraphs of $G$ whose connected components are induced subgraphs.  On the right below is the lattice of flats corresponding to the cycle matroid of the graph shown on the left.
   \begin{center}   
\edgelabeledgraph{
    -1/0/1//-120,
    -1/2/2//120,
    1/2/3//60,
    1/0/4//-60
    }{
    1/2/0.5/above/a,
    2/3/0.5/above/b,
    4/3/0.5/above/e,
    1/4/0.5/below/d,
    1/3/0.5/above/c}
    \qquad
        \lattice{0/0/\emptyset/0,
    -2/1/a/1,
    -1/1/b/2,
    0/1/c/3,
    1/1/d/4,
    2/1/e/5,
    -2.5/2/abc/6,
    2.5/2/cde/7,
    -1.5/2/ad/8,
    -0.5/2/ae/9,
    0.5/2/bd/10,
    1.5/2/be/11,
    0/3/abcde/12}{0/1,0/2,0/3,0/4,0/5,1/6,2/6,3/6,3/7,4/7,5/7,1/8,4/8,1/9,5/9,2/10,4/10,2/11,5/11,6/12,7/12,8/12,9/12,10/12,11/12}
\end{center} 
\end{example}

An element $e \in E$ is a \term{loop} of $M$ if the set $\{e\}$ is dependent.  If $e,f \in E$ are not loops, then $e$ and $f$ are \term{parallel} if $\{e,f\}$ is dependent.  A matroid is \term{simple} if it has no loops and no pairs of parallel elements.  For clarity of notation, we only consider the graded M\"obius algebras of simple matroids.  However, for any matroid $M$, there is a unique (up to isomorphism) simple matroid whose lattice of flats is isomorphic to $\L(M)$, so there is no loss of generality by imposing this restriction.  This matroid, called the \term{simplification} of $M$ and denoted by $\si(M)$, is the matroid on the set of rank-one flats of $M$ such that a set of flats $\{Y_1, \dots,Y_t\}$ is independent if and only if $\rk_M(Y_1 \vee \cdots \vee Y_t) = t$.  

Let $M$ be a matroid  on a ground set $E$ and $X \subseteq E$ be a subset.  There are two constructions for producing new matroids from $M$ that will play an important role in the subsequent sections:  \vspace{1 ex}
\begin{itemize}
\item The \term{restriction} of $M$ to $X$ is the matroid $M\vert X$ on the ground set $X$ whose independents sets are precisely the independent sets of $M$ that are contained in $X$.  The flats of $M\vert X$ are of the form $F \cap X$ for some flat $F$ of $M$.  In particular, when $F$ is a flat of $M$, every flat of $M\vert F$ is also a flat of $M$ so that $M \vert F$ is a matroid quotient of $M$, and the lattice of flats of $M\vert F$ is just the interval $[\emptyset, F]$ in $\L(M)$.  \vspace{1 ex}
\item The \term{contraction} of $M$ by $X$ is the matroid $M/X$ on the ground set $E \setminus X$ whose independent sets consist of all subsets $Y \subseteq E \setminus X$ such that $Y \cup B$ is independent in $M$ for some (or equivalently, every) basis $B$ of $M\vert X$.  If $F$ is a flat of $M$, then $G \subseteq E \setminus F$ is a flat of $M/F$ if and only if $G \cup F$ is a flat of $M$ so that $M/F$ is a matroid quotient of $M$, and the lattice of flats of $M/F$ is isomorphic to the interval $[F, E]$ in $\L(M)$. \vspace{1 ex}
\end{itemize} 
We refer the reader to \cite{Welsh} and \cite{Oxley} for further details about these constructions.

A common theme in the study of matroids involves using a total order $\prec$ on the ground set of a matroid $M$ to shed light on its structure.  Given such a total order and a set $X \subseteq E$, we denote by $\min X$ the smallest element of $X$ in the chosen order.  Sets of the form $C \setminus \min C$, where $C$ is a circuit of $M$ are called \term{broken circuits} of $M$.  A set $X \subseteq E$ that does not contain any broken circuits is called an \term{nbc-set} (because they contain no broken circuits).  The collection of all nbc-sets $\BC(M, \prec)$ is easily seen to be a pure subcomplex of the simplicial complex of independent sets of $M$, called the \term{broken circuit complex} of $M$.  We refer the reader to \cite{homology:and:shellability:of:matroids} for more details about broken circuit complexes.

\subsection{Hilbert Functions and Free Resolutions}

Fix a field $\kk$. Let $A = \bigoplus_{i \ge 0}A_i$ be a  standard graded $\kk$-algebra with maximal ideal $\m = \bigoplus_{i > 0} A_i$, and let $N = \bigoplus_{i \in \ZZ} N_i$ be a finitely generated, graded $A$-module.  The \term{Hilbert function} of $N$ is defined as $\HF_N(i) := \dim_\kk N_i$, and its generating function is the \term{Hilbert series} of $N$, denoted by 
\[\HS_N(t) := \sum_{i \geq 0} \HF_N(i) t^i.\] 
The module $N$ has a \text{minimal graded free resolution} $\mathbf{F}$ over $A$, which is an exact sequence of the form
\[\mathbf{F}: \ \ \ \cdots \xrightarrow{} F_i \xrightarrow{\partial_{i}} F_{i-1} \xrightarrow{} \cdots \xrightarrow{} F_1 \xrightarrow{\partial_1} F_0,\]
where $N \iso \coker(\partial_1)$, each $F_i$ is a finite-rank graded free $A$-module, and all maps are graded homomorphisms of degree $0$ with $\partial_i(F_i) \subseteq \m F_{i-1}$.  We can write $F_i \iso \bigoplus_{j} A(-j)^{\beta^A_{i,j}(N)}$, where $A(-j)$ denotes the rank-one free $A$-module generated in degree $j$.  The numbers $\beta^A_{i,j}(N)$ are the \term{graded Betti numbers} of $N$ over $A$ and their generating function
\[ \P_N^A(s, t) = \sum_{i \ge 0} \beta_{i, j}^A(N)s^jt^i \]
is called the \term{graded Poincar\'{e} series} of $N$ over $A$. It is convenient to display the graded Betti numbers as a table, called the \term{graded Betti table} of $N$, in which $\beta_{i,i+j}^A(N)$ is placed in column $i$ and row $j$;  see Section~\ref{open:problems} for an example.  The series $\P^A_N(t) = \P^A_N(1, t)$ is called the \term{Poincar\'{e} series} of $N$, and its coefficients
\[ \beta_i^A(N) := \sum_j \beta_{i,j}^A(N)\]
are the \term{(total) Betti numbers} of $N$.  When $N = \kk$, it is common to omit the subscript and refer to $\P^A(t) := \P^A_{\kk}(t)$ as the Poincar\'{e} series of $A$. 
We refer the reader to \cite{graded:syzygies} for further details about free resolutions and their numerical invariants.

Being a Koszul algebra forces restrictions on the Betti numbers and Hilbert series of $A$ that do not hold for all quadratic algebras.  In particular, $A$ being Koszul is equivalent to $\P^A_{\kk}(s, t)$ being a power series in $st$.  Also, $A$ is Koszul if and only if
\[\HS_A(t) \P^A(-t) = 1.\]

\section{Quadracity and Gr\"obner Bases} \label{quadracity:and:GB's}

\noindent
In this section, we study the graded M\"obius algebra of an arbitrary matroid $M$.
We give a standard graded presentation for $\GMA{M}$ along with universal and lexicographic Gr\"obner bases of its defining ideal.  

\subsection{Graded M\"{o}bius Algebras} 

\noindent
By \cite[2.15]{semi-small:decompositions}, the graded M\"obius algebra $\GMA{M}$ embeds as a subalgebra of the augmented Chow ring $\Chow(M)$ of $M$.  The linear forms of $\Chow(M)$ have a basis consisting of elements $y_i$ for each $i \in E$ and $x_F$ for each flat $F$ of $M$.  If $I \subseteq E$, we set $y_I = \prod_{i \in I} y_i$ in $\Chow(M)$.  We can then identify $\GMA{M}$ with a subalgebra of $\Chow(M)$ by mapping each basis element $y_F$ of $\GMA{M}$ to the monomial $y_I$ in $\Chow(M)$, where $I$ is any independent set with $\cl(I) = F$.

 The statement about the universal Gr\"{o}bner basis in the following proposition first appeared in \cite[3.3]{Maeno:Numata:published}; we give a shorter proof of this fact.

\begin{prop} \label{GMA:presentation}
Let $M$ be a simple matroid.  Let $U$ be the Stanley-Reisner ideal of $M$ as a simplicial complex in the polynomial ring $S = \kk[y_i \mid i \in E]$, $J$ be the ideal generated by all binomials of the form $y_I - y_{I'}$ for all independent sets $I$ and $I'$ of $M$ with $\cl(I) = \cl(I')$, and $L$ be the ideal generated by all binomials $y_{C \setminus i} - y_{C \setminus j}$ for all circuits $C$ of $M$ and all $i, j \in C$.  Then:
\begin{enumerate}[label = \textnormal{(\alph*)}]
\item We have a presentation $\GMA{M} \iso S/Q$, where $Q = (y_i^2 \mid i \in E) + U +  J$.  Moreover, the generators of $Q$ are a universal Gr\"obner basis. 

\item $Q = (y_i^2 \mid i \in E) + L$, and the generators of the latter ideal are a Gr\"obner basis for $Q$ with respect to every lex ordering for any ordering of the elements of $E$.
\end{enumerate}
\end{prop}

\begin{proof}
(a)   By Lemma 2.9 and Proposition 2.15 in \cite{semi-small:decompositions}, it follows that $\GMA{M}$ is a quotient of the ring $S/Q$.  The ring $S/Q$ is spanned by squarefree monomials corresponding to the independent sets of $M$.  Moreover, monomials corresponding to independent sets with the same closure are identified in $S/Q$.  Since the Hilbert function of $\GMA{M}$ just counts to the number of flats of $M$ of a given rank by construction, it follows that the surjection $S/Q \to \GMA{M}$ must be an isomorphism.

Let $>$ be any monomial order on $S$.  For each flat $F$ of $M$, let $I_F$ denote the independent set contained in $F$ such that $y_{I_F}$ is minimal among all monomials $y_I$ with $\cl(I) = F$ in the chosen monomial order.  Then for each such $I$ with $I \neq I_{F}$, we have $y_I - y_{I_F} \in J$ so that $y_I \in \init_>(Q)$.  If $N$ denotes the monomial ideal generated by all monomials $y_i^2$ for $i \in E$, $y_D$ for $D$ a dependent set, and $y_{I'}$ for $I'$ an independent set with $y_{I'} > y_{I_{\cl(I')}}$, then $S/N$ has $S/\init_>(Q)$ as a quotient, and both rings have the same Hilbert function so that $N = \init_>(Q)$, which shows that the generators of $Q$ are a Gr\"obner basis with respect to $>$.  

(b)  By the previous part, it suffices to show that the leading monomial of each generator of $U$ and $J$ is divisible by a leading monomial of $L$.  First, we note that every monomial $y_D$ for $D$ a dependent set is divisible by the leading monomial of the binomial $y_{C \setminus i} - y_{C \setminus j}$ in $L$ for any circuit $C \subseteq D$ and any $i, j \in C$.  On the other hand, suppose $y_I - y_{I'}$ is a binomial generator of $J$ with $y_I > y_{I'}$.  If $\abs{I \setminus I'} = 1$, let $I' \setminus I = \{i\}$ and $I \setminus I' = \{j\}$, and note that there is a circuit $C \subseteq I \cup I'$ since $\cl(I) = \cl(I')$.  Note also that $i \in C$ since otherwise we would have that $C \subseteq I$, contradicting that $I$ is independent.  Similarly, we must have $j \in C$ so that $y_I - y_I' = y_U(y_{C \setminus i} - y_{C \setminus j}) \in L$, where $U = (I \cup I') \setminus C$ and $y_{C \setminus i}$ is necessarily the leading monomial of $y_{C \setminus i} - y_{C \setminus j} \in L$.    If $\abs{I \setminus I'} \geq 2$, then by applying \cite[7.3.2]{homology:and:shellability:of:matroids} to the bases of the restriction of $M$ to the flat $\cl(I) = \cl(I')$, we see that there is an independent set $I''$ with $\cl(I'') = \cl(I)$ such that $y_I > y_{I''}$ and $\abs{I \setminus I''} = 1$, and so, it follows that $y_I$, which is the leading monomial of $y_I - y_{I''} \in J$ is divisible by the leading monomial of some generator of $L$.  
\end{proof}

\begin{cor}
If $M$ is a matroid with $\rk M \leq 2$, then $\GMA{M}$ is G-quadratic.
\end{cor}

\begin{proof}
If $\rk M = 1$, then $M \iso U_{1,1}$ since $M$ is simple, and so, we have $\GMA{M} \iso \kk[y_1]/(y_1^2)$ by the preceding proposition since there are no circuits.  If $\rk M = 2$, then $M \iso U_{2, n}$ for some $n \geq 2$.  It follows that $\GMA{M}$ is quadratic since all circuits have size 3, and the $h$-vector of $\GMA{M}$ is $(1, n, 1)$.  Thus, $\GMA{M}$ has a Gr\"obner flag by \cite[2.12]{Groebner:flags} and, hence, is G-quadratic by \cite[2.4]{Groebner:flags}.
\end{proof}

\begin{example} \label{GMA:of:broken:3-trampoline}
Let $M = M(G)$ be the cycle matroid of the graph $G$ shown below.
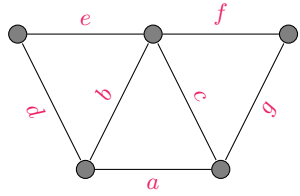
\begin{center}
\edgelabeledgraph{
    -1/0/v_1//-120,
    0/2/v_2//90,
    1/0/v_3//-60,
    -2/2/w_1//120,
    2/2/w_2//60
    }{
    v_1/v_2/0.5/above/b, 
    v_1/v_3/0.5/below/a, 
    v_2/v_3/0.5/above/c,
    v_1/w_1/0.5/below/d,
    w_1/v_2/0.5/above/e,
    w_2/v_2/0.5/above/f,
    w_2/v_3/0.5/below/g}
\end{center}
Abusing notation slightly, the graded M\"obius algebra of $M$ has a presentation
\[ 
\GMA{M} \iso \frac{\kk[a, b, c, d, e, f, g]}{(a^2, b^2, c^2, d^2, e^2, f^2, g^2) + L}
\]
where \vspace{1 ex}
\[
L = (ab - ac, ab - bc, bd - be, bd - de, cf - cg, cf - fg). \vspace{1 ex}
\]
Note that the larger cycles in $G$ such as the 4-cycle $\{a, d, e, c\}$ do not contribute minimal generators to the ideal $L$ since they have chords.  For example, we have \vspace{1 ex}
\[
ade - ace = -a(bd - de)  + a(bd - be) + e(ab - ac). \vspace{1 ex}
\]
\end{example}

As the above example illustrates, in studying when the defining ideals of graded M\"obius algebras are generated by quadratic relations, we are naturally led to consider various notions of what it means for a circuit of a matroid to have a chord.  While all of these notions agree in the case of graphic matroids, there does not seem to be much agreement on the correct notion of chordality for general matroids.  Our terminology follows \cite{chordality:in:matroids} and \cite{supersolvable:saturated:matroids} rather than \cite{chordal:graphs:and:binary:supersolvable:matroids} to distinguish these different notions.  

\begin{defn}
Let $M$ be a simple matroid.  We say that $M$ is: \vspace{1 ex}
\begin{itemize}
    \item \term{C-chordal} if for every circuit $C$ of $M$ of size at least four there is an element $e \in E$ and circuits $A, B$ of $M$ such that $A \cap B = \{e\}$ and $C = (A \setminus e) \sqcup (B \setminus e)$, \vspace{1 ex}
    \item \term{T-chordal} if for every circuit $C$ of $M$ of size at least four there is an element $w \in E \setminus C$ and elements $u, v \in C$ such that $\{u, v, w\}$ is a circuit, and \vspace{1 ex}
    \item \term{line-closed} if whenever $F \subseteq E$ such that for all $i, j \in F$ we have $\cl (i, j) \subseteq F$, then $F$ is a flat of $M$.
\end{itemize}
\end{defn}    

\begin{prop} \label{C-chordal:implies:quadratic:GMA} \label{line-closed:implies:T-chordal} \label{quadratic:GMA:implies:T-chordal}
Let $M$ be a simple matroid. 
\begin{enumerate}[label = \textnormal{(\alph*)}]
    \item If $M$ is C-chordal, then $\GMA{M}$ is a quadratic algebra.
    \item If $\GMA{M}$ is a quadratic algebra, then $M$ is T-chordal.
    \item If $M$ is line-closed, then $M$ is T-chordal.
\end{enumerate}
\end{prop}

\begin{proof}
(a) It suffices to show that the ideal $L$ of the preceding proposition equals the ideal $L'$ generated by the quadratic binomials in $L$.  Let $C$ be a circuit of $M$ with $\abs{C} \geq 4$, and let $i, j \in C$.  Then by assumption, there is an $e \in E$ and circuits $A, B$ of $M$ such that $A \cap B = \{e\}$ and $C = (A \setminus e) \sqcup (B \setminus e)$.  Suppose without loss of generality that $i \in A$.  Since $M$ is a simple matroid, we know that every circuit has size at least three so that $\abs{A}, \abs{B} < \abs{C}$.  Since $y_{C \setminus i} - y_{C \setminus j} = (y_{C \setminus i} - y_{C \setminus \ell}) - (y_{C \setminus \ell} - y_{C \setminus j})$ for any $\ell \in B \setminus e$, it suffices to assume $j \in B$ and show that $y_{C \setminus i} - y_{C \setminus j} \in L'$.  In that case, we note that $B \setminus e \subseteq C \setminus i$ and $A \setminus e \subseteq C \setminus j$ so that 
\[
y_{C \setminus i} - y_{C \setminus j} = y_{A \setminus \{i, e\}}(y_{B \setminus e} - y_{B \setminus j}) + y_{B \setminus \{j, e\}}(y_{A \setminus i} - y_{A \setminus e}) \in L'
\]
by a simple induction on the size of $C$.

(b) Let $Q$ be the defining ideal of the graded M\"obius algebra $\GMA{M}$ as in Proposition~\ref{GMA:presentation}.  We note that the quadratic generators of $Q$ consist of the squares $y_i^2$ for each $i \in E$ and the binomials $y_I - y_{I'}$, where $I$ and $I'$ are independent sets of $M$ of size 2 with $\cl(I) = \cl(I')$. If $\GMA{M}$ is quadratic, then these polynomials generate $Q$.  Let $C$ be a circuit of $M$ of size at least four, and let $i, j \in C$.  Since $y_{C \setminus i} - y_{C \setminus j} \in Q$, we can write 
$$y_{C \setminus i} - y_{C \setminus j} = \sum_{p = 1}^r c_pm_p(y_{I_p} - y_{J_p}),$$ where $c_p \in \kk$, $m_p$ is a monomial, and $I_p, J_p$ are distinct independent sets of size 2 with $\cl(I_p) = \cl(J_p)$.  As $y_{C \setminus i}$ must belong to the support of one of the monomials on the right side of this equality, after possibly reordering the terms in this sum and switching the roles of $I_p$ and $J_p$, we may assume that $m_1y_{I_1} = y_{C \setminus i}$. Write $I_1 = \{u, v\} \subseteq C$.  Since $J_1 \neq I_1$, there is a $w \in J_1 \setminus I_1$, and because $w \in \cl(J_1) = \cl(I_1)$ and $M$ is simple, it follows that $\{u, v, w\}$ is a circuit.  In particular, we note that $w \notin C$ or else $C$ would not be a minimal dependent set.  Thus, $M$ is T-chordal.  

(c) Suppose that $M$ is line-closed, and let $C$ be a circuit of $M$ with $\abs{C} \geq 4$ and $i \in C$.  First, note that $C \setminus i$ is an independent set which is not a flat since $i \in \cl(C \setminus i)$.  Hence, there exist $u, v \in C \setminus i$ such that $\cl \{u, v\}  \nsubseteq C \setminus i$.  Choose $w \in \cl\{u, v\} \setminus (C \setminus i)$, and set $T = \{u, v, w\}$.  Note that every 2-element subset of $T$ is independent since $M$ is a simple matroid, and so, $T$ is a circuit.  Moreover, $w \notin C$ since otherwise we would have $w = i$ so that $T \subseteq C$, contradicting that $C$ is a circuit. Thus, $M$ is T-chordal.
\end{proof}

\subsection{Connections with Orlik-Solomon Algebras} 

\noindent
The \term{Orlik-Solomon algebra} of a simple matroid $M$ with a total order $\prec$ on its ground set $E$ is the quotient $\OS{M} = S/J$ of the exterior algebra $S = \exterior_\kk\langle z_i \mid i \in E \rangle$ by the ideal 
\[
J = (\dd(z_C) \mid C \; \text{is a circuit}),
\]
where $z_I := z_{i_1}z_{i_2} \cdots z_{i_s}$ for each $I = \{i_1 \prec i_2 \prec \cdots \prec i_s\} \subseteq E$ and
\[
\dd(z_I) = \sum_{p = 1}^s (-1)^{p-1}z_{I \setminus i_p}.  
\]

\begin{example} \label{OS:of:broken:3-trampoline}
The Orlik-Solomon algebra of the matroid $M$ from Example \ref{GMA:of:broken:3-trampoline} is \vspace{1 ex}
\[ 
\OS{M} = \frac{\exterior_\kk \langle a, b, c, d, e, f, g \rangle}{(bc - ac + ab, de - be + bd, fg - cg + cf)}. \vspace{1 ex}
\]
\end{example}

Our presentation of the graded M\"obius algebra bears a strong resemblance to that of the Orlik-Solomon algebra; the non-monomial relations of the graded M\"obius algebra come from splitting up the relations of the Orlik-Solomon algebra into binomials.  This does not appear to be a purely coincidental equational similarity.  Just as the graded M\"obius algebra of a linear matroid is isomorphic to the (even-dimensional) cohomology ring of the variety obtained by taking the closure of the linear subspace it determines in a product of projective lines as studied by Ardila and Boocher \cite{Ardila-Boocher:Schubert:varieties}, when $\kk = \CC$ and $M$ is the matroid associated with a central complex hyperplane arrangement $\mathcal{H}$ in $\CC^d$, it is well-known that $\OS{M}$ is isomorphic to the de Rham cohomology ring of the complement $\CC^d \setminus \bigcup_{H \in \mathcal{H}} H$ \cite{OS}.  
Furthermore, Shelton and Yuzvinsky observed that hyperplane arrangements with Koszul Orlik-Solomon algebras satisfy the lower central series formula \cite{SY97}, and work of Papadima and Yuzvinsky \cite{PY99} showed that an arrangement has a Koszul Orlik-Solomon algebra exactly when its complement is a rational $K(\pi,1)$-space.  
As a result, much of our work investigating when graded M\"obius algebras are Koszul is motivated by drawing parallels with the much better studied case of Orlik-Solomon algebras.  For example, part (a) of Proposition \ref{C-chordal:implies:quadratic:GMA} is already similar in spirit to \cite[3.3]{Irena:hyperplane:arrangements} for Orlik-Solomon algebras.

We recall the following theorem characterizing when the Orlik-Solomon algebra of a matroid has a quadratic Gr\"{o}bner basis, which the reader should compare with our Theorem~\ref{GMA:quadratic:GB}.  

\begin{thm}[{\cite[2.8]{broken:circuit:complexes}}, {\cite[4.2]{Irena:hyperplane:arrangements}}]
Let $M$ be a simple matroid and $\OS{M} = S/J$ be its Orlik-Solomon algebra.  Then the following are equivalent:
\begin{enumerate}[label = \textnormal{(\alph*)}]

\item There is a monomial order $>$ such that $\init_>(J)$ is quadratic.

\item There is a lex order $>_\mathrm{lex}$ such that $\init_{>_\mathrm{lex}}(J)$ is quadratic.

\item $\L(M)$ is supersolvable.
\end{enumerate}
\end{thm}

It is clear that the leading terms of the forms $\dd(z_C)$ generating $J$ (with respect to the lexicographic order induced by the total order on $E$) are precisely the monomials $z_{C \setminus \min C}$ corresponding to broken circuits of $M$.  Bj\"{o}rner showed that these monomials generate the initial ideal of $J$ \cite[7.10.1]{homology:and:shellability:of:matroids} and, thus, the monomials corresponding to nbc-sets constitute a monomial basis for $\OS{M}$.  The equivalence (b) $\Iff$ (c) of the preceding theorem was proved by Bj\"orner and Ziegler, who showed combinatorially that $\L(M)$ being supersolvable is equivalent to the minimal broken circuits of $M$ all having size 2, and the equivalence (a) $\Iff$ (b) was proved by Peeva, who showed that every initial ideal of the Orlik-Solomon ideal is the initial ideal with respect to some lex order.

As an immediate consequence of the above theorem, when the lattice of flats $\L(M)$ is supersolvable, $\OS{M}$ has a quadratic Gr\"obner basis and, hence, is Koszul.  In general, it is an open question whether every Koszul Orlik-Solomon algebra comes from a supersolvable matroid.  However, in the case when $M = M(G)$ is the cycle matroid of a graph $G$, we have that $\OS{M}$ is Koszul if and only if $G$ is chordal if and only if $\mathcal{L}(M)$ is supersolvable \cite{SS02, Stanley72}. 

\subsection{Quadratic Gr\"obner Bases}

Turning our attention now to the issue of when the graded M\"obius algebra of a matroid has a quadratic Gr\"obner basis, we make the main combinatorial definition of the paper, motivated by the results of the next section.

\begin{defn} \label{strongly:T-chordal}
Let $M$ be a simple matroid, and let $\succ$ be a total order on its ground set $E$.  We say that: 
\begin{itemize}
    \item A set $S \subseteq E$ has a \term{MAT-triple} if there exist $u, v \in S$ and a $w \in E$ such that $\{u, v, w\}$ is a circuit and $w \succ \min(u, v)$. \vspace{1 ex}
    \item A circuit $C$ of $M$ is a \term{MAT-circuit} if for every $i \in C \setminus \min C$, the set $C \setminus i$ has a MAT-triple. \vspace{1 ex}
    \end{itemize} 
     The matroid $M$ is \term{strongly T-chordal} if there is a total order $\succ$ on $E$ such that every circuit $C$ of $M$ of size at least four is a MAT-circuit, in which case we call $\succ$ a \term{strong elimination order} for $M$.  When $M = M(G)$ is the cycle matroid of a graph $G$ and $M$ is strongly $T$-chordal, we call the associated ordering on the edges of $G$ a \term{strong edge elimination order}.
\end{defn}

Here, MAT is short for Multiple Addition Theorem relating to the construction of free hyperplane arrangements.  See Section~\ref{MAT-labelings} for more information.

\begin{thm} \label{GMA:quadratic:GB}
Let $M$ be a simple matroid and $\GMA{M} \iso S/Q$ be its graded M\"obius algebra  as in Proposition~\ref{GMA:presentation}.  Then the following are equivalent:
\begin{enumerate}[label = \textnormal{(\alph*)}]

\item There is a monomial order $>$ such that $\init_>(Q)$ is quadratic.

\item $M$ is strongly T-chordal.

\item There is a lex order $>_\mathrm{lex}$ such that $\init_{>_\mathrm{lex}}(Q)$ is quadratic.

\item $\GMA{M}$ is quadratic, and there exists a total order $\succ$ on $E$ such that for every circuit $C$ of $M$ of size exactly four is a MAT-circuit.
\end{enumerate}
\end{thm}

\begin{proof}
It is obvious that (c) implies (a).  Once we know that (b) implies (c), it is also immediate that (b) implies (d).

(a) $\Implies$ (b): Define a total order $\prec$ on $E$ by $u \prec v$ if and only if $y_u > y_v$.    We claim that $\prec$ is a strong elimination order for $M$.  If $C$ is a circuit of $M$ of size at least 4, we can write $C = \{i_1 \prec i_2 \prec \cdots \prec i_r \}$, and we note that $i_j \prec i_k$ if and only if
\[ 
y_{C \setminus i_k} = y_{C \setminus \{i_j, i_k\}}y_{i_j} > y_{C \setminus \{i_j, i_k\}}y_{i_k} = y_{C \setminus i_j}. 
\]
We must show that $C \setminus i_k$ has a MAT-triple for each $k > 1$.  

For each $k > 1$, we know that the binomial $y_{C \setminus i_k} - y_{C \setminus \min C}$ is part of a universal Gr\"obner basis for $Q$ by Proposition~\ref{GMA:presentation}, and so, it follows that $y_{C \setminus i_k} \in \init_>(Q)$.  Since $\init_>(Q)$ is quadratic, there must be a quadratic binomial generator $y_I - y_{I'}$ of $Q$ whose leading monomial $y_I$ divides $y_{C \setminus i_k}$.  Write $I = \{u, v\}$ and $I' = \{w, t\}$, and note that either $u \preceq w, t$ or $v \preceq w, t$, since otherwise we would have
$w, t \prec u, v$ so that
\[ y_{I'} = y_wy_t > y_uy_t > y_uy_v = y_I, \]
contradicting that $y_I$ is the leading monomial.  And so, after possibly relabeling, we may assume without loss of generality that $u \prec w$.  Either way, there is an element $w \in I' \setminus I$ such that $w \succ \min(u, v)$.  Since $\cl(I) = \cl(I')$ and $M$ is simple, it follows that $\{u, v, w\}$ is a circuit.  Moreover, $u, v \in C \setminus i_k$, since $y_I$ divides $y_{C \setminus i_k}$.  Thus, $T$ is a MAT-triple for $C \setminus i_k$, and $M$ is strongly T-chordal.

(b) $\Implies$ (c): Let $\succ$ be a strong elimination order for $M$, and let $>_\mathrm{lex}$ be the lex order determined by the variable order defined by $y_u > y_v$ if and only if $u \prec v$. In this case, part (b) of Proposition~\ref{GMA:presentation} implies that the polynomials $y_i^2$ for $i \in E$ and $y_{C \setminus i} - y_{C \setminus \min C}$ for each circuit $C$ of $M$ and $i \neq \min C$ form a Gr\"obner basis for the ideal $Q$. Let $C$ be a circuit of $M$ of size at least four and $i \in C \setminus \min C$.  As $M$ is strongly T-chordal, we know that there are $u, v \in C \setminus i$ and $w \in E \setminus C$ such that  $\{u, v, w\}$ is a circuit with $w \succ \min(u, v)$.  Without loss of generality, we may assume that $u \succ v$. Then $y_uy_v - y_wy_v \in Q$ with leading term $y_uy_v$ that divides $y_{C \setminus i}$.  As this holds for every circuit $C$ of size at least four and every $i \in C \setminus \min C$, it follows that $\init_{>_\mathrm{lex}}(Q)$ is quadratic.

(d) $\Implies$ (c): Let $\succ$ be the total order on $E$ as in part (d) of the proposition, and let $>_\mathrm{lex}$ denote the lex order determined by the variable order $y_u > y_v$ if and only if $u \prec v$.  Again, part (b) of Proposition~\ref{GMA:presentation} implies that the polynomials $y_i^2$ for $i \in E$ and $y_{C \setminus i} - y_{C \setminus \min C}$ for each circuit $C$ of $M$ and $i \neq \min C$ form a Gr\"obner basis for the ideal $Q$, and a similar argument to the proof of the previous implication shows that $y_{C \setminus i}$ is not a minimal generator of $\init_{>_\mathrm{lex}}(Q)$ for every circuit $C$ of size four and $i \in C \setminus \min C$.  Hence, $\init_{>_\mathrm{lex}}(Q)$ has no minimal generators of degree three.  Since by assumption $Q$ is generated by quadrics, it then follows from \cite[34.13]{graded:syzygies} that the quadratic generators of $Q$ are a Gr\"obner basis.
\end{proof}

Figure \ref{quadracity:implications} below summarizes the connections between the   various combinatorial properties of $M$ and quadracity properties of $\GMA{M}$ and $\OS{M}$ discussed. We conclude this section with a few examples related to certain implications in the figure.

\begin{figure}[htb]
    \begin{center}
    \begin{tikzcd}[column sep = 5em, row sep = 3.5 em]
    \text{supersolvable}
    \dar[Leftrightarrow, swap]{\text{\cite[2.8]{broken:circuit:complexes}\,}} 
    &
    \substack{ \text{\normalsize binary and} \\ \text{\normalsize supersolvable}} 
    \dar[Rightarrow, swap]{\text{\cite[2.2]{chordal:graphs:and:binary:supersolvable:matroids}\,}} 
    \lar[Rightarrow]
    &
    \text{strongly T-chordal} 
    \dar[Leftrightarrow]{\text{\; Thm \ref{GMA:quadratic:GB}}}
    \\
    \OS{M} \,\text{quadratic GB}
    \dar[Rightarrow]
    & 
    \text{C-chordal}  
    \arrow[Rightarrow]{ld}{\text{\cite[3.3]{Irena:hyperplane:arrangements}\,}}
    \arrow[Rightarrow, swap]{rd}{\text{Prop \ref{C-chordal:implies:quadratic:GMA}(a)\; }}
    &
    \GMA{M} \, \text{quadratic GB} \dar[Rightarrow]
    \\
    \OS{M} \,\text{quadratic} 
    \dar[Rightarrow,swap]{\text{\cite[6.10]{Yuzvinsky:Orlik-Solomon:algebras:survey}}}
    &
    &
    \GMA{M} \, \text{quadratic} \dar[Rightarrow]{\text{\; Prop \ref{quadratic:GMA:implies:T-chordal}(b)}} 
    \\
    \text{3-independent} \rar[Rightarrow]{\text{\cite[6.8]{Yuzvinsky:Orlik-Solomon:algebras:survey}}} 
    &
    \text{line-closed} \rar[Rightarrow]{\text{Prop \ref{line-closed:implies:T-chordal}(c)}}
    &
    \text{T-chordal}
    \end{tikzcd}
    \end{center}
    \caption{Matroid Chordality and Quadracity Properties of Graded M\"obius and Orlik-Solomon Algebras}\label{fig1}
    \label{quadracity:implications}
\end{figure}
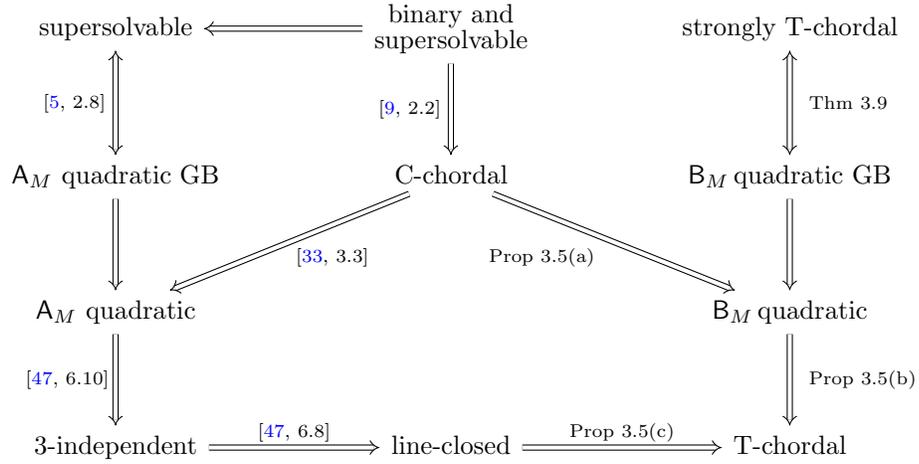

\begin{example}
We give an example that $M$ being T-chordal does not imply that $\GMA{M}$ is quadratic, which shows the converse to Proposition~\ref{quadratic:GMA:implies:T-chordal}(b) is false in general.  Probert shows that the matroid $L_{2,3}$ from \cite[Figure 6.4]{chordality:in:matroids}, which is obtained from the uniform matroid $U_{3,5}$ on the ground set $E = \{1,2,3,4,5\}$ by removing the basis $234$, is T-chordal but not C-chordal.  The defining ideal of $\GMA{L_{2,3}}$ is $Q = (y_1^2,y_2^2,y_3^2,y_4^2,y_5^2) + L$, where
    \[ L = \left(y_{3}y_{5}-y_{4}y_{5}, \
    y_{3}y_{4}-y_{4}y_{5},
    \ y_{1}y_{4}y_{5}-y_{2}y_{4}y_{5}, \ y_{1}y_{2}y_{5}-y_{2}y_{4}y_{5}, \ y_{1}y_{2}y_{4}-y_{2}y_{4}y_{5}, \ y_{1}y_{2}y_{3}-y_{2}y_{4}y_{5}\right). \]
    In particular, $\GMA{L_{2,3}}$ is not quadratic.
\end{example}

\begin{example}
We give an example that
$M$ being T-chordal does not imply that $M$ is line-closed, and thus, the converse to Proposition~\ref{line-closed:implies:T-chordal}(c) is false in general.
    Let $\mathcal{W}^3$ denote the rank-three whirl matroid.  The circuits of $\mathcal{W}^3$ of size $3$ are $123, 345, 156$, while those of size $4$ are 1245, 1246, 1346, 2346, 2356, 2456.
    It is straightforward to check that $\mathcal{W}^3$ is T-chordal.  However, Falk shows that $\mathcal{W}^3$ is not line-closed \cite[2.6]{Falk02}.\footnote{To avoid any confusion for the reader, we note that Falk mistakenly refers to this matroid as the rank-three wheel matroid $\mathcal{W}_3$;  however, the rank-three wheel is the same as the cycle matroid of the complete graph $K_4$, which is C-chordal and, hence, line-closed.  The whirl $\mathcal{W}^3$ is the unique relaxation of $\mathcal{W}_3$.} 
\end{example}

\begin{example}
We give an example that quadracity of the graded M\"obius algebra of a matroid $M$ does not imply that $M$ is C-chordal, and hence, the converse to Proposition~\ref{C-chordal:implies:quadratic:GMA}(a) is false in general.
Consider the Betsy Ross matroid $B$ pictured in Figure~\ref{betsyross}.  One can check that $\GMA{B}$ is quadratic, however $B$ is not C-chordal.  To see this, observe that there are 25 circuits of size 3: 5 corresponding to 3 points on a line through the center and 20 corresponding to 3 of the 4 points on a line connecting points of the star.  The set $C = 0123$ is a circuit of size $4$.  If $B$ were C-chordal, we could partition $C$ into two pairs of points such that each pair lie on a line in the diagram and the intersection of those two lines corresponded to a point in the matroid; clearly this is not possible.  

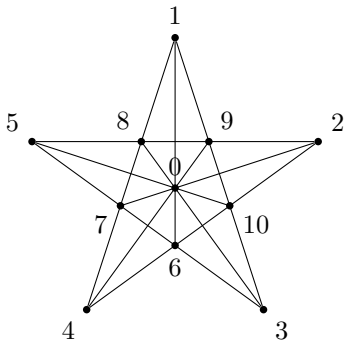
\begin{figure}[htb]
    \centering
    \begin{tikzpicture}
    \node[label=+90:{$0$},circle,fill=black,scale=.3] (Z) at (0,0) {};
    \node[label=+90:{$1$},circle,fill=black,scale=.3] (A) at ({2*cos(90)},{2*sin(90)}) {};
    \node[label=+162:{$5$},circle,fill=black,scale=.3] (E) at ({2*cos(90+72)},{2*sin(90+72)}) {};
    \node[label=+234:{$4$},circle,fill=black,scale=.3] (D) at ({2*cos(90+144)},{2*sin(90+144)}) {};
    \node[label=+306:{$3$},circle,fill=black,scale=.3] (C) at ({2*cos(90+216)},{2*sin(90+216)}) {};
    \node[label=+18:{$2$},circle,fill=black,scale=.3] (B) at ({2*cos(90+288)},{2*sin(90+288)}) {};
\draw [name path=A--C] (A) -- (C);
\draw [name path=B--D] (B) -- (D);
\draw [name path=C--E] (C) -- (E);
\draw [name path=D--A] (D) -- (A);
\draw [name path=E--B] (E) -- (B);
\path [name intersections={of=A--C and B--D,by=F}];
\node [label=+342:{$10$},circle,fill=black,scale=.3] at (F) {};
\path [name intersections={of=B--D and C--E,by=G}];
\node [label=+270:{$6$},circle,fill=black,scale=.3] at (G) {};
\path [name intersections={of=C--E and D--A,by=H}];
\node [label=+198:{$7$},circle,fill=black,scale=.3] at (H) {};
\path [name intersections={of=D--A and E--B,by=I}];
\node [label=+126:{$8$},circle,fill=black,scale=.3] at (I) {};
\path [name intersections={of=E--B and A--C,by=J}];
\node [label=+54:{$9$},circle,fill=black,scale=.3] at (J) {};
\draw [name path=A--C] (A) -- (G);
\draw [name path=B--D] (B) -- (H);
\draw [name path=C--E] (C) -- (I);
\draw [name path=D--A] (D) -- (J);
\draw [name path=E--B] (E) -- (F);
    \end{tikzpicture}
    \caption{The Betsy Ross matroid}
    \label{betsyross}
\end{figure}
\end{example}

In the remaining sections of the paper, we specialize to studying the graded M\"obius algebras of graphic matroids. Below we observe that there exist non-graphic matroids which are strongly T-chordal and, hence, have Koszul graded M\"obius algebras.

\begin{example}
The  Fano matroid $F_7$ is the projective geometry of points in the projective plane over $\ZZ/2\ZZ$ shown below with points being represented by strings of their homogeneous coordinates.
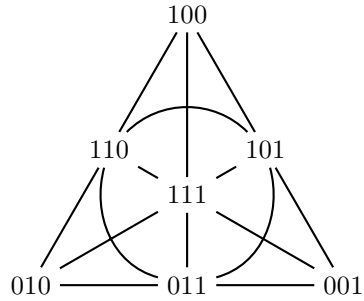
\begin{center}
    \begin{tikzpicture}[scale = 0.6]
        \node (7) at (0,0) {111};
        \node (1) at (90: 4) {100};
        \node (2) at (210:4) {010};
        \node (3) at (-30:4) {001};
        \node (4) at (0, -2) {011};
        \node (5) at (30:2) {101};
        \node (6) at (150: 2) {110};
        \draw[thick] (4) -- (7) -- (1) -- (6) -- (2) -- (4) -- (3) -- (5) -- (1);
        \draw[thick] (6) -- (7) -- (3);
        \draw[thick] (2) -- (7) -- (5) to[out = -75, in = 15] (4) to[out = 165, in = -105] (6) to[out = 45, in = 135] (5);
    \end{tikzpicture}
\end{center}
Its circuits of size 3 are precisely the sets of points lying on each of the 7 lines, while its circuits of size at least four consist of 4 points in general linear position or, equivalently, the points not on a given line.  It is not too hard to check that the ordering of the points $$100 \prec 010 \prec 001 \prec 011 \prec 101 \prec 110 \prec 111$$ is a strong elimination order so that $F_7$ is strongly T-chordal. 
\end{example}

\section{Strongly Chordal Graphs and MAT-Labelings}
\label{MAT-labelings}

\noindent
Let $G$ denote a finite simple graph with vertex set $V(G)$ and edge set $E(G)$. Given vertices $v, w \in V(G)$, we write $vw$ to denote that the set $\{v, w\}$ is an edge of $G$.  The set of \term{neighbors} of $v$ in $G$ is denoted by $N(v)$, and the \term{closed neighborhood} of $v$ is the set $N[v] = N(v) \cup \{v\}$.  To simplify notation, we will also write $\GMA{G}$ to denote the graded M\"obius algebra of the cycle matroid $M(G)$.  We refer the reader to \cite{West:graph:theory} for any unexplained graph-theoretic terminology.

Recall that a graph $G$ is called \term{chordal} if every cycle of length at least four in $G$ has a chord. Chordal graphs can be characterized as the graphs for which every induced subgraph has a \term{simplicial vertex}, a vertex whose closed neighborhood within the subgraph forms a clique. Equivalently, a graph $G$ is chordal if and only if there is an ordering $v_1, \dots, v_n$ of its vertices such that $v_i$ is a simplicial vertex of the induced subgraph $G[v_i, v_{i+1}, \dots, v_n]$ for all $i$. Such a vertex ordering is called a \term{perfect elimination order}.  

\subsection{Strongly Chordal Graphs} A graph $G$ is called \term{strongly chordal} if there is an ordering $v_1, \dots, v_n$ of its vertices such that for all $i < j$ and $k < \ell$, if $v_k, v_\ell \in N[v_i]$ and $v_j \in N[v_k]$, then $v_j \in N[v_{\ell}]$.  Such a vertex ordering is called a \term{strong (perfect) elimination order}.  If it becomes necessary to disambiguate between the strong elimination order on the vertices of a strongly chordal graph $G$ and a strong elimination order as defined in Definition \ref{strongly:T-chordal} for the ground set of the cycle matroid $M(G)$, which is the set of edges of $G$ by definition, we will refer to strong \textit{vertex} elimination orders and strong \textit{edge} elimination orders respectively. 
 
A vertex $v$ of $G$ is called \term{simple} if the collection of distinct sets in $\{N[w] \mid w \in N[v]\}$ is totally ordered by inclusion.  In particular, every simple vertex is simplicial.  If $G$ is strongly chordal with strong elimination order $v_1,\ldots,v_n$, then it is easily seen that the vertex $v_i$ is a simple vertex of the induced subgraph $G[v_i, v_{i+1}, \dots, v_n]$ for all $i$.  Such a vertex ordering is called a \term{simple elimination order}. 

A graph $T$ is an \term{$n$-trampoline} (sometimes also called an \term{$n$-sun}) if it has vertices $v_1,\dots, v_n, w_1, \dots, w_n$ for some $n \geq 3$ and edges $v_iv_j$ for all $i \neq j$, $w_iv_i$ for all $i$, $w_iv_{i+1}$ for all $i < n$, and $w_nv_1$.  Figure \ref{4-trampoline} below shows the 4-trampoline.

 \begin{figure}[htb]
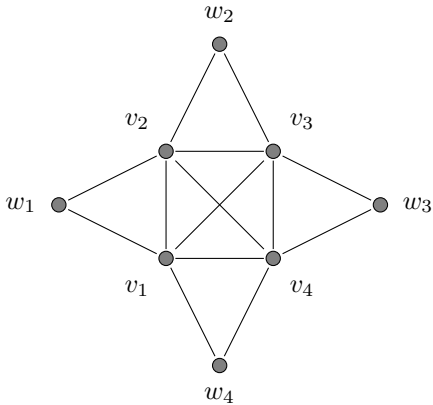

 \centering
 \captionsetup{justification=centering}
\graph[scale = 0.95]{
    -1/0/v_1/v_1/-120,
    -1/2/v_2/v_2/120,
    1/2/v_3/v_3/60,
    1/0/v_4/v_4/-60,
    -3/1/w_1/w_1/180,
    0/4/w_2/w_2/90,
    3/1/w_3/w_3/0,
    0/-2/w_4/w_4/-90
    }{
    v_1/v_2,
    v_2/v_3,
    v_3/v_4,
    v_1/v_4,
    v_2/v_4,
    v_1/v_3,
    w_1/v_1,
    w_1/v_2,
    w_2/v_2,
    w_2/v_3,
    w_3/v_3,
    w_3/v_4,
    v_1/w_4,
    v_4/w_4}
     \caption{The $4$-trampoline}
     \label{4-trampoline}
 \end{figure}

\noindent Trampoline graphs are chordal, since it is easily seen that $w_1, \dots, w_n, v_1, \dots, v_n$ is a perfect elimination order for the $n$-trampoline.  On the other hand, setting $w_{i + n} = w_i$, we note that $w_i \in N[v_{i+1}] \setminus N[v_{i+2}]$ and $w_{i+2} \in N[v_{i+2}] \setminus N[v_{i+1}]$ for all $i$ since $n \geq 3$.  Consequently, the sets $N[v_{i+1}]$ and $N[v_{i+2}]$ are incomparable for all $i$ so that no vertex of the $n$-trampoline is simple, and so, no trampoline graph is strongly chordal. 

The following characterization of strongly chordal graphs is due to Farber.

\begin{thm}[{\cite[3.3, 4.1]{Farber:strongly:chordal:graphs}}] \label{Farber} \label{strongly:chordal:graphs}
For a graph $G$, the following are equivalent:
\begin{enumerate}[label = \textnormal{(\alph*)}]
    \item $G$ is strongly chordal.
    \item $G$ has a simple elimination order.
    \item Every induced subgraph of $G$ has a simple vertex.
    \item $G$ is chordal and has no induced $n$-trampoline for any $n \geq 3$.
\end{enumerate}
\end{thm}

\subsection{MAT-Labelings}

The purpose of this subsection is to connect the notion of a strongly chordal graph with our definition of a strongly T-chordal matroid.  By definition, the cycle matroid of a graph $G$ is strongly T-chordal if there is an ordering of the edges of $G$ satisfying the properties in Definition~\ref{strongly:T-chordal}.  Although edge orderings are much less common in graph theory than vertex orderings, recent work of Tran and Tsujie shows that strongly chordal graphs are precisely the graphs that admit a certain type of edge weighting called a MAT-labeling \cite[4.10, 5.12]{MAT:labelings}, which gives a partial ordering of the edges. 

\begin{defn}
    Let $G$ be a graph and $\lambda: E(G) \to \ZZ_{> 0}$ be a labeling of its edges.  Set $\pi_k = \lambda^{-1}(k)$  and $E_k = \bigcup_{j \leq k} \pi_j$ for each $k \in \ZZ_{> 0}$, and put $E_0 = \emptyset$.  We say that $\lambda$ is a \term{MAT-labeling} of $G$ if the following conditions hold for all $k \in \ZZ_{> 0}$: \vspace{1 ex}
    \begin{enumerate}[label = (ML\arabic*)]
        \item $\pi_k$ is a forest. \vspace{1 ex}
        \item $\cl(\pi_k) \cap E_{k-1} = \emptyset$, where $\cl(\pi_k)$ denotes the closure of $\pi_k$ in the cycle matroid of $G$. \vspace{1 ex}
        \item Every $e \in \pi_k$ forms exactly $k - 1$ triangles (3-cycles) with edges in $E_{k-1}$.
    \end{enumerate}
\end{defn}

MAT-labelings can also be characterized by vertex orderings via the notion of a MAT-simplicial vertex.  The reader may want to compare the condition \ref{MS3} below with our definition of a MAT-triple in Definition \ref{strongly:T-chordal}.

\begin{defn}
    Given a graph $G$ with edge labeling $\lambda: E(G) \to \ZZ_{> 0}$, a vertex $w$ of $G$ is called \term{MAT-simplicial} if: \vspace{1 ex}
    \begin{enumerate}[label = (MS\arabic*)]
        \item $w$ is a simplicial vertex of $G$. \vspace{1 ex}
        \item $\{\lambda(vw) \mid v \in N(w)\} =\{1, 2, \dots, \ell\}$, where $\ell = \abs{N(w)}$. \vspace{1 ex}
        \item For all distinct $u, v \in N(w)$, $\lambda(uv) < \max(\lambda(uw), \lambda(vw))$. \label{MS3}
    \end{enumerate}
\end{defn}

The following proposition summarizes the essential facts that we will need about MAT-labelings.

\begin{prop}\label{MAT:labelings:for:strongly:chordal:graphs}
Let $G$ be a strongly chordal graph with MAT-labeling $\lambda$, and let $K$ be a maximal clique of $G$ with $\abs{K} = \ell$.  Write $E(G) = \pi_1 \sqcup \pi_2 \sqcup \cdots \sqcup \pi_c$ for some $c$.  Then: 
\begin{enumerate}[label = \textnormal{(\alph*)}]
\item $\abs{\pi_k \cap K} = \ell - k$ if $k \leq \ell - 1$, and $\abs{\pi_k \cap K} = 0$ otherwise.  Hence, $c = \omega(G) - 1$ where $\omega(G)$ is the clique number of $G$
\item There is a unique edge $e_K$ contained in $K$ with $\lambda(e_K) = \ell - 1$, and $K$ is the unique maximal clique containing $e_K$.
\end{enumerate}
\end{prop}

\begin{proof}
(a) Since $K$ is a maximal clique of $G$, $\lambda\vert_K$ is a MAT-labeling of $G[K]$ by \cite[4.9]{MAT:labelings}.  The first statement then follows from \cite[2.6, 4.4]{MAT:labelings}.  Applying this observation to all maximal cliques of $G$ gives $c = \omega(G) - 1$.

(b) It is immediate from part (a) that there is exactly one edge $e_K$ contained in $K$ with $\lambda(e_K) = \ell - 1$.  Suppose that there is another maximal clique $K'$ that also contains $e_K$.  Then $e_K$ is contained in $K \cap K'$.  Since $\lambda\vert_{K \cap K'}$ is a MAT-labeling of $G[K \cap K']$ by \cite[4.9]{MAT:labelings}, it follows from part (a) that 
\[ \lambda(e_K) \leq \abs{K \cap K'} - 1 < \ell - 1 = \lambda(e_K) \]
which is a contradiction.  Hence, $K$ must be the only maximal clique containing $e_K$.
\end{proof}

We now show that strongly chordal graphs admit strong edge elimination orders.  The converse is given in Theorem~\ref{SEEO:implies:strongly:chordal}.

\begin{thm} \label{strongly:chordal:implies:strongly:T-chordal}
Let $G$ be a strongly chordal graph with MAT-labeling $\lambda$, and let $\prec$ be any total order on $E(G)$ refining the partial order given by $e < e'$ if and only if $\lambda(e) > \lambda(e')$.  Then $\prec$ is a strong edge elimination order for $G$.
\end{thm}

\begin{proof}
We proceed by induction on $\omega(G)$.  If $\omega(G) \leq 2$, then $\prec$ is trivially a strong elimination order since $E(G) \subseteq \pi_1$ is a forest.  So, we may assume that $\omega(G) \geq 3$ and that the theorem holds for all strongly chordal graphs with strictly smaller clique number.

By Proposition~\ref{MAT:labelings:for:strongly:chordal:graphs}, it is clear that $\lambda$ restricts to a MAT-labeling of the graph $G' = G \setminus \pi_{\omega(G) - 1}$ so that $G'$ is strongly chordal with clique number strictly smaller than $\omega(G)$.  Hence, $\prec$ restricts to a strong edge elimination order on $G'$ by induction.  

 As $G$ is chordal, $B_{M(G)}$ is quadratic by Proposition~\ref{C-chordal:implies:quadratic:GMA}(a).  Thus by Theorem~\ref{GMA:quadratic:GB}, it suffices to show that any $4$-cycle is a MAT-circuit. 
Let $C$ be a 4-cycle in $G$.  If $C$ is contained in $G'$, there is nothing to prove, so we may assume that at least one of the edges $e$ of $C$ is in $\pi_{\omega(G) - 1}$.  Without loss of generality, we may assume that $e = \min_\prec C$.

{\sc Case (1):} Suppose that the vertices of $C$ form a clique of $G$.  In that case, $C$ is contained in a maximal clique $K$ of $G$, necessarily of size $\omega(G)$.  Since $e$ is the unique edge of $K$ with $\lambda(e) = \omega(G) - 1$, it follows that every other edge $a$ of $K$ has label less than $e$ so that $e \prec a$ by Proposition~\ref{MAT:labelings:for:strongly:chordal:graphs}.  Hence, it is easily seen that $C \setminus a$ always has a MAT-triple for every $a \in C \setminus e$.

{\sc Case (2):} Suppose that the vertices of $C$ do not form a clique of $G$.  As $G$ is chordal, the induced subgraph on the vertices of $C$ must have the form shown below.
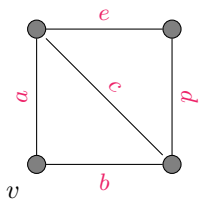
\begin{center}
\edgelabeledgraph{
    -1/0/2/v/-120,
    -1/2/1//120,
    1/2/3//60,
    1/0/4//-60
    }{
    2/1/0.5/above/a,
    4/3/0.5/above/d,
    1/4/0.5/above/c,
    2/4/0.5/below/b,
    1/3/0.5/above/e}
\end{center}
As in the previous case, we have $e \prec c$ so that $C \setminus a$ and $C \setminus b$ both have $\{c, d, e\}$ as a MAT-triple.  It remains to show that $c \neq \min\{a, b, c\}$.  We will show that $\lambda(c) < \max\{\lambda(a), \lambda(b)\}$.

Let $K$ and $K'$ be maximal cliques containing $\{a, b, c\}$ and $\{c, d, e\}$ respectively. By \cite[4.9]{MAT:labelings}, we know that $\lambda$ restricts to a MAT-labeling on $G[K]$, $G[K']$, and $G[K \cap K']$.  Thus, $\lambda$ restricts to a MAT-labeling on $H = G[K \cup K']$ by \cite[5.8]{MAT:labelings}.  Since $H$ is not a clique, the proof of \cite[5.2]{MAT:labelings} shows that there is a MAT-simplicial vertex $w$ for $H$ in $V(H) \setminus K$.  If $w = v$, then $\lambda(c) < \max\{\lambda(a), \lambda(b)\}$ as wanted.  Otherwise, if $w \neq v$, we can replace $H$ with $H \setminus w$ to obtain a strictly smaller induced subgraph containing $K \cup v$, which is not a clique and on which $\lambda$ restricts to a MAT-labeling by \cite[5.3]{MAT:labelings}.  Again, we see that $H$ must have a MAT-simplicial vertex in $V(H) \setminus K$, and so, this process eventually terminates with an induced subgraph in which $v$ is MAT-simplicial so that $\lambda(c) < \max\{\lambda(a), \lambda(b)\}$ as wanted.
\end{proof}

\begin{cor} \label{broken:trampolines:are:G-quadratic}
If $G$ is a strongly chordal graph, then the graded M\"obius algebra $\GMA{G}$ of the cycle matroid $M(G)$ has a quadratic Gr\"obner basis.  Hence, $\GMA{G}$ is a Koszul algebra.
\end{cor}

\begin{example} \label{broken:4-trampoline:is:strongly:T-chordal}
Consider the graph $G$ shown below.  It can easily be checked that the labels on the edges of $G$ form a MAT-labeling.
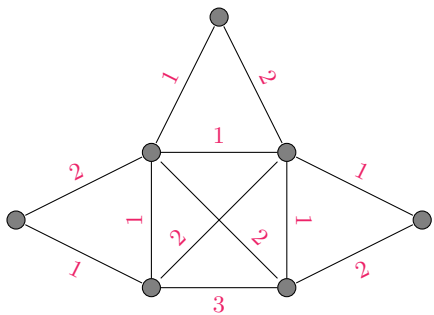
\begin{center}
\edgelabeledgraph{
    -1/0/v_1//-120,
    -1/2/v_2//120,
    1/2/v_3//60,
    1/0/v_4//-60,
    -3/1/w_1//180,
    0/4/w_2//60,
    3/1/w_3//0
    }{
    v_1/v_2/0.5/above/1, 
    v_2/v_3/0.5/above/1,
    v_4/v_3/0.5/above/1,
    v_1/v_4/0.5/below/3,
    v_2/v_4/0.75/above/2,
    v_1/v_3/0.25/above/2,
    w_1/v_1/0.5/below/1,
    w_1/v_2/0.5/above/2,
    w_2/v_2/0.5/above/1,
    w_2/v_3/0.5/above/2,
    w_3/v_3/0.5/above/1,
    w_3/v_4/0.5/below/2}
\end{center}

Hence, $G$ has a strong edge elimination order    
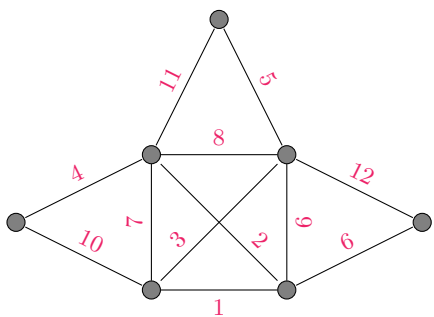
\begin{center}
\edgelabeledgraph{
    -1/0/v_1//-120,
    -1/2/v_2//120,
    1/2/v_3//60,
    1/0/v_4//-60,
    -3/1/w_1//180,
    0/4/w_2//60,
    3/1/w_3//0
    }{
    v_1/v_2/0.5/above/7, 
    v_2/v_3/0.5/above/8,
    v_4/v_3/0.5/above/9,
    v_1/v_4/0.5/below/1,
    v_2/v_4/0.75/above/2,
    v_1/v_3/0.25/above/3,
    w_1/v_1/0.5/above/10,
    w_1/v_2/0.5/above/4,
    w_2/v_2/0.5/above/11,
    w_2/v_3/0.5/above/5,
    w_3/v_3/0.5/above/12,
    w_3/v_4/0.5/above/6}
\end{center}
where $1 \prec 2 \prec 3 \prec \cdots \prec 12$, as guaranteed by the above theorem.
\end{example}

Here we prove that the converse to Theorem~\ref{strongly:chordal:implies:strongly:T-chordal} holds and thus obtain a new characterization of strongly chordal graphs.  While this already follows via the characterization of Koszul graded M\"obius algebras, we include a direct proof here.

\begin{thm}
\label{SEEO:implies:strongly:chordal} If $G$ is a graph with a strong edge elimination order $\prec$, then $G$ is strongly chordal.   
\end{thm}

\begin{proof}
    By the definition of strong edge elimination orders, it follows immediately that $G$ is chordal.  Suppose $G$ is not strongly chordal.  Then by Theorem~\ref{Farber}, $G$ has an induced subgraph $T$ isomorphic to an $n$-trampoline for some $n \ge 3$.  Using the previous notation, let $v_1,\ldots,v_n,w_1,\ldots,w_n$ denote the vertices of $T$ so that $v_1,\ldots,v_n$ induce a clique in $T$, and consider the cycle $C = \{v_1v_2,\  \dots,\  v_{n-1}v_n, \ v_1v_n\}$.  
    
    We claim that $v_1v_n = \min C$. We consider two cases.  Suppose first that $n = 3$ and that the vertices of $B = T \setminus w_3$ are labeled as shown below.  
    \begin{center}
\edgelabeledgraph{
    -1/0/v_1/v_1/-120,
    0/2/v_2/v_2/90,
    1/0/v_3/v_3/-60,
    -2/2/w_1/w_1/120,
    2/2/w_2/w_2/60
    }{
    v_1/v_2/0.5/above/, 
    v_1/v_3/0.5/below/, 
    v_2/v_3/0.5/above/,
    v_1/w_1/0.5/below/,
    w_1/v_2/0.5/above/,
    w_2/v_2/0.5/above/,
    w_2/v_3/0.5/below/}
\end{center}
    If the claim does not hold, then without loss of generality we may assume that $v_1v_2 \prec v_1v_3, v_2v_3$.  Consider the 4-cycle $C' = \{v_1w_1, w_1v_2, v_2v_3, v_1v_3\}$.  At least one of the edges $w_1v_2,w_1v_1$ is not $\min C'$.  It follows that $C'$ is not a MAT-circuit, and we have a contradiction. Hence, the claim holds when $n = 3$.

 Suppose now that $n \geq 4$.  If the edge $e = v_1v_n$ satisfies $e \neq \min C$, then there must exist edges $u, v \in C \setminus e$ and $w \notin C$ such that $\{u, v, w\}$ is a 3-cycle and $w \succ \min(u, v)$.  It follows without loss of generality that $u = v_iv_{i+1}$, $v = v_{i+1}v_{i+2}$, and $w = v_iv_{i+2}$ for some $i < n - 1$.  Since $G[v_i,  v_{i+1},  v_{i+2},  w_i,  w_{i+1}]$ is isomorphic to the graph $B$ above, it follows from the preceding paragraph that $w \prec u, v$, which is a contradiction.  Hence, we must have $e = \min C$ as claimed.

 However, a completely analogous argument shows that $v_1v_2 = \min C$, which is a contradiction as $n \geq 3$.  Therefore, $G$ must be strongly chordal.
\end{proof}

\section{Graded M\"obius Algebras of Strongly Chordal Graphs}\label{proof:of:main:thm}

 \noindent
In this section, we study the Koszul property of graded M\"obius algebras of graphic matroids, and finish the proof of Theorem~A.

Clearly, if $G$ is chordal, then the cycle matroid $M(G)$ is C-chordal, and likewise, if $M(G)$ is T-chordal, then $G$ is chordal.  And so, as an immediate consequence of Proposition \ref{C-chordal:implies:quadratic:GMA} we have the following.

\begin{thm} \label{quadratic:GMA:iff:chordal}
For a graph $G$, the graded M\"obius algebra $\GMA{G}$ is quadratic if and only if $G$ is chordal.
\end{thm}

However, not all chordal graphs have a graded M\"{o}bius algebra that is Koszul.  To characterize the Koszul property for graded M\"{o}bius algebras of graphic matroids, a stronger notion of chordality is needed; it turns out that strongly chordal graphs provide exactly the right notion.

\begin{thm} \label{Koszul:imples:strongly:chordal}
For a graph $G$, 
if $\GMA{G}$ is Koszul then $G$ is strongly chordal.
\end{thm}

Recall that strongly chordal graphs can be characterized as those chordal graphs that do not contain an induced trampoline.  If $T$ is an $n$-trampoline with vertices labeled $v_1, \dots, v_n, w_1, \dots, w_n$ as in Section \ref{MAT-labelings}, we call a graph isomorphic to $T \setminus w_n$ a \term{broken $n$-trampoline}.  For example, the broken 3-trampoline is the graph of Example~\ref{GMA:of:broken:3-trampoline}, and the broken 4-trampoline is the graph of Example~\ref{broken:4-trampoline:is:strongly:T-chordal}.  It is easily seen that the vertex order $w_1, \dots, w_{n-1}, v_1, \dots, v_n$ is a simple elimination order for the broken trampoline $T \setminus w_n$.  Hence, broken trampolines are strongly chordal.  They will play an important role in the proof of Theorem \ref{Koszul:imples:strongly:chordal}.  

\begin{lemma} \label{GMAs:and:contraction}
Let $\GMA{M} \iso S/Q$ as in Proposition \ref{GMA:presentation} and $a \in E$.  Then:
\begin{enumerate}[label = \textnormal{(\alph*)}]
    \item $(0 :_{\GMA{M}} y_a)$ is generated by $y_a$ and the binomials $y_{C \setminus i} - y_{C \setminus j}$ for all circuits $C$ of $M/a$ and all $i, j \in C$.  
    \item $\GMA{M}/(0 :_{\GMA{M}} y_a) \iso \GMA{N}$, where $N = \si(M/a)$, the simplification of $M/a$.
    \item If $M$ is C-chordal, then 
    \[ (0 :_{\GMA{M}} y_a) = (y_a) + ( y_j - y_i \mid \{a, i, j\} \;\text{a circuit of $M$}). \]
\end{enumerate}
\end{lemma}

\begin{proof}
(a) First, we will show that $(0:_{\GMA{M}} y_a)$ is generated by $y_a$ and the binomials $y_I - y_{I'}$, where $I, I'$ are independent sets of $M/a$ with $\cl_{M/a}(I) = \cl_{M/a}(I')$.  Clearly, $(0:_{\GMA{M}} y_a)$ contains each of the preceding elements since $y_a^2 = 0$ in $\GMA{M}$, and if $I$ and $I'$ are independent sets of $M/a$ with $\cl_{M/a}(I) = \cl_{M/a}(I')$, then $I \cup a$ and $I' \cup a$ are independent sets of $M$ with equal closures in $M$ by \cite[3.1.8, 3.1.12]{Oxley} so that $y_a(y_I - y_{I'}) = y_{I \cup a} - y_{I' \cup a} = 0$ in $\GMA{M}$.  Conversely, if for each flat $F$ of $M$ we choose an independent set $I_F$ with $\cl(I_F) = F$, we know that the monomials $y_{I_F}$ form a $\kk$-basis for $\GMA{M}$.  Let $f \in {(0:_{\GMA{M}} y_a)}$.  We may assume that $f$ is homogeneous of degree $r$ so that $f = \sum_{\rk F = r} c_Fy_{I_F}$ for some $c_F \in \kk$. It follows that 
\[ 
0 = y_af = \sum_{\substack{\rk G = r + 1 \\ a \in G}} \, \Big(\sum_{F \vee a = G} c_F\Big)y_{I_G}
\]
so that $\sum_{F \vee a = G} c_F = 0$ for each flat $G$ of rank $r + 1$ with $a \in G$.  Consequently, we have:
\begin{align*}
    f &= \sum_{a \in F} c_F y_{I_F} + \sum_{a \notin F} c_Fy_{I_F} \\
    &= \sum_{a \in F} c_F y_{I_F} + \sum_{\substack{\rank G = r + 1 \\ a \in G}}\Big(\sum_{F \vee a = G} c_Fy_{I_F}\Big).
\end{align*}
For each flat $F$ in the sum on the left above, we can choose a basis for $F$ of the form $I \cup a$ so that $y_{I_F} = y_Iy_a$ in $\GMA{M}$.  On the other hand, for each flat $G$ in the sum on the right above, we can choose a designated flat $F_0$ of rank $r$ with $a \notin F_0$ and $F_0 \vee a = G$, and it is easily seen from the fact that $\sum_{F \vee a = G} c_F = 0$ that the sum $\sum_{F \vee a = G} c_Fy_{I_F}$ is a sum of binomials of the form $y_{I_F} - y_{I_{F_0}}$ for each flat $F$ with $F \vee a = G$ with $F \neq F_0$.  In addition, for each such flat $F$, $a \notin F$ implies $I_F \cup a$ is independent so that $I_F$ is independent in $M/a$, and $$\cl_M(I_F \cup a) = F \vee a = F_0 \vee a = \cl_M(I_{F_0} \cup a)$$ implies $\cl_{M/a}(I_F) = \cl_{M/a}(I_{F_0})$.

From the preceding paragraph, we know that $(0 :_{\GMA{M}} y_a)$ contains all of the binomials $y_{C \setminus i} - y_{C \setminus j}$ for all circuits $C$ of $M/a$ and all $i, j \in C$.  On the other hand, an argument similar to the proof of part (b) of Proposition \ref{GMA:presentation} shows that $(0 :_{\GMA{M}} y_a)$ is contained in the ideal generated by $y_a$ and such circuit binomials.

(b)  We first recall that $N = \si(M/a)$ is a simple matroid whose ground set is the set of rank 1 flats of $M/a$.  Consider the surjective algebra map $\pi: S \to \GMA{N}$ given by sending $y_i \mapsto y_{\cl_{M/a}(i)}$ for $i \in E \setminus a$ and $y_a \mapsto 0$.  We will show that $\pi(Q) = 0$ so that there is an induced homomorphism $\GMA{M} \to \GMA{N}$.  Clearly, we have $\pi(y_a^2) = 0$ and $\pi(y_i^2) = y_{\cl_{M/a}(i)}^2 = 0$ for all $i \in E \setminus a$.  Let $I$ be an independent set of $M$.  If $a \in \cl_M(I)$, then either $a \in I$ so that $\pi(y_I) = 0$, or $I$ is a dependent set of $M/a$ so that $\rk_{M/a} I < \abs{I}$.  In the latter case, let $\si(I) = \{ \cl_{M/a}(i) \mid i \in I \}$, and note that $\pi(y_I)$ is divisible by $y_{\si(I)}$.  If $\si(I)$ is dependent in $N$, then $y_{\si(I)} = 0$ in $\GMA{N}$ by Proposition \ref{GMA:presentation}, so we may further assume that $\si(I)$ is independent.  This implies that 
\[ \abs{\si(I)} = \rk_{M/a} \bigvee_{b \in \si(I)} b \leq \rk_{M/a} \cl_{M/a}(I) = \rk_{M/a} I < \abs{I}
\]
which is only possible if there are $i, j \in I$ such that $\cl_{M/a}(i) = \cl_{M/a}(j)$.  But then $\pi(y_I) = 0$ since it is divisible by $y_{\cl_{M/a}(i)}^2$.

Now, consider a binomial $y_I - y_{I'}$ in $S$, where $I, I'$ are independent sets of $M$ with $\cl_M(I) = \cl_M(I')$.  If $a \in \cl_M(I)$, the preceding paragraph shows that $\pi(y_I - y_{I'}) = 0$, so we may assume that $a \notin \cl_M(I)$.  In that case, $I$ and $I'$ are both independent sets of $M/a$ with $\cl_{M/a}(I) = \cl_{M/a}(I')$.  In that case, we note that $$\bigvee_{b \in \si(I)} b = \cl_{M/a}(\bigcup_{b \in \si(I)} b) = \cl_{M/a}(I)$$ in the lattice of flats of $M/a$ so that
\[ 
\abs{I} = \rk \cl_{M/a}(I) = \rk_{M/a} \bigvee_{b \in \si(I)} b \leq \abs{\si(I)}.
\]
As there is an obvious surjection of $I \to \si(I)$ by taking closures, it follows that the above inequality must be an equality so that $\si(I)$ is also independent in $N$, and in particular, the map $I \to \si(I)$ is a bijection.  Furthermore, we have $\cl_N(\si(I)) = \{ \cl_{M/a}(i) \mid i \in \cl_{M/a}(I)\}$.  Therefore $\pi(y_I - y_{I'}) = y_{\si(I)} - y_{\si(I')} = 0$, since $\cl_N(\si(I)) = \cl_N(\si(I'))$.  Hence, we have an induced homomorphism $\bar{\pi}: \GMA{M} \to \GMA{N}$, and moreover, the preceding argument shows that $\bar{\pi}((0 :_{\GMA{M}} y_a)) = 0$ so that we have an induced surjection $\GMA{M}/(0:_{\GMA{M}} y_a) \to \GMA{N}$.

We know that $\GMA{M}$ is spanned by the monomials $y_I$ for each independent set $I$ of $M$.  If $a \in \cl(I)$, then there is an independent set $I'$ with $a \in I'$ and $\cl(I) = \cl(I')$ so that $y_I = y_{I'}$ is divisible $y_a$ and, hence, is zero in $\GMA{M}/(0 :_{\GMA{M}} y_a)$.  As a result, we see that $\GMA{M}/(0 :_{\GMA{M}} y_a)$ is spanned by the monomials $y_I$ with $a \notin \cl(I)$.  Each such $I$ is also an independent set of $M/a$, and moreover, any two such monomials are identified in $\GMA{M}/(0 :_{\GMA{M}} y_a)$ if they have the same closure in $M/a$.  Thus, $\GMA{M}/(0 :_{\GMA{M}} y_a)$ has a spanning set of monomials corresponding bijectively with flats of $M/a$.  Since $M/a$ and $N = \si(M/a)$ have isomorphic lattices of flats \cite[1.7.5]{Oxley} and $\GMA{N}$ has a monomial basis corresponding to the flats of $N$, it follows that the map $\GMA{M}/(0 :_{\GMA{M}} y_a) \to \GMA{N}$ is an isomorphism.

(c) If $\{a, i, j\}$ is a circuit of $M$, then $\{i, j\}$ is a circuit of $M/a$ by \cite[3.1.11]{Oxley} so that $(0 :_{\GMA{M}} y_a)$ contains all of the linear forms $y_i - y_j$.  Conversely, let $C$ be a circuit of $M/a$.  Then either $C$ or $C \cup a$ is a circuit of $M$.  We must show for any $i, j \in C$, the binomial $y_{C \setminus i} - y_{C \setminus j}$ belongs to the ideal $L = (y_a) + (y_j - y_i \mid \{a, i, j\} \;\text{a circuit of $M$})$.  If $C$ is also a circuit of $M$, then $y_{C \setminus i} - y_{C \setminus j} = 0$ in $\GMA{M}$, so we may assume that $C \cup a$ is a circuit of $M$.  If $\abs{C \cup a} = 3$, there is nothing to prove, so we may assume $\abs{C \cup a} \geq 4$.  Since $M$ is C-chordal, we know there is an $e \notin C \cup a$ and circuits $A, B$ such that $A \cap B = \{e\}$ and $C \cup a = (A \setminus e) \sqcup (B \setminus e)$.  Since $M$ is a simple matroid, every circuit of $M$ has size at least three so that $\abs{A}, \abs{B} < \abs{C \cup a}$.  Suppose without loss of generality that $i \in A$.  As $\abs{B} \geq 2$, there is an $\ell \in B \setminus a \subseteq C$, and we have $$y_{C \setminus i} - y_{C \setminus j} = (y_{C \setminus i} - y_{C \setminus \ell}) - (y_{C \setminus \ell} - y_{C \setminus j}).$$  Consequently, it suffices to assume that $j \in B$.  Furthermore, we may assume without loss of generality that $a \in A$ so that $A \setminus a$ contains a circuit $D$ of $M/a$.  Note that $B   \setminus e \subseteq C \setminus i$ and $A \setminus \{e, a\} \subseteq C \setminus j$.  Additionally, we must have $e \in D$ since otherwise it would follow that $D \subseteq A \setminus \{e, a\} \subsetneq C$, contradicting that $C$ is a circuit of $M/a$. Similarly, we must have $i \in D$ since otherwise it would follow that $D \subseteq A \setminus {a, i} \subseteq C \setminus i$, contradicting that $C \setminus i$ is an independent set of $M/a$.  Then $D \setminus e \subseteq C \setminus j$, and we have
\begin{align*}
y_{C \setminus i} - y_{C \setminus j} &= y_{A \setminus \{i,e\}}(y_{B \setminus e} - y_{B \setminus j}) + y_{B\setminus \{j, e\}}y_{A \setminus (D \cup a)}(y_{D \setminus i} - y_{D \setminus e}) \\
&= y_{B\setminus \{j, e\}}y_{A \setminus (D \cup a)}(y_{D \setminus i} - y_{D \setminus e}) 
\end{align*}
in $\GMA{M}$, where $\abs{D \cup a} \leq \abs{A} < \abs{C \cup a}$.  Thus, $y_{C \setminus i} - y_{C \setminus j} \in L$ by an induction on the size of $C \cup a$.  
\end{proof}

\begin{lemma}\label{poincare:lemma}
Let $T$ be the $n$-trampoline and $B$ be the broken $n$-trampoline for some $n \geq 3$.  Set $a = v_1v_n$, $b = v_1w_n$, and $c = v_nw_n$.  Then:
\begin{enumerate}[label = \textnormal{(\alph*)}]
\item There is an algebra retract $\GMA{T} \onto \GMA{B}$ with kernel $(y_b, y_a - y_c) = (0 :_{\GMA{T}} y_b)$.
\item The Poincar\'e series satisfy $$\P^{\GMA{T}}(s, t) =  \dfrac{\P^{\GMA{B}}(s, t)}{1-st\Big(1 + \P^{\GMA{B}}_{\GMA{B}/(y_a)}(s, t)\Big)}.$$
\end{enumerate}
\end{lemma}

In the proof we will use
the tool  of a large homomorphism.  Levin \cite{large:homomorphisms} defines a surjective ring homomorphism $A \to A'$ to be \term{large} if the induced homomorphism $\Tor_*^A(\kk, \kk) \to \Tor^{A'}_*(\kk, \kk)$ is surjective.  Critically, for us, he shows that this is equivalent to having a factorization of Poincar\'{e} series
\[
\P^A(s, t) = \P^A_{A'}(s, t)\P^{A'}(s, t).
\]

\begin{proof}
(a) First, we observe that the $n$-trampoline $T$ is chordal so that $\GMA{T}$ is quadratic by Theorem~\ref{quadratic:GMA:iff:chordal}. 

Let $S_T = \kk[y_i \mid i \in E(T)]$ and $S_B = \kk[y_i \mid i \in E(B)]$, and denote the defining ideals of $\GMA{T}$ and $\GMA{B}$ in $S_T$ and $S_B$ respectively by $Q_T$ and $Q_B$.  Since $\GMA{T}$ is quadratic, we know $Q_T$ is generated by the squares $y_i^2$ for each $i \in E(T)$ and the binomials $y_{C \setminus i} - y_{C \setminus j}$ for each 3-cycle $C$ of $T$ and $i, j \in C$.   As $\{a, b, c\}$ is the only 3-cycle of $T$ not contained in $B$, it follows that 
\begin{equation} \label{defining:ideal:of:trampoline}
Q_T = Q_B S_T + (y_b^2, y_c^2, y_b(y_a - y_c), y_a(y_b - y_c)).
\end{equation}

Hence, we have a natural injection $\GMA{B} \to \GMA{T}$ which admits a retract $\pi: \GMA{T} \to \GMA{B}$ by sending $y_b \mapsto 0, y_c \mapsto y_a$.

It remains to show that $\ker \pi = (y_b, y_a - y_c) = (0 :_{\GMA{T}} y_b)$.   Since $(Q_T, y_b, y_a - y_c) = Q_BS_T + (y_b, y_a - y_c)$, it is clear that $\GMA{T}/(y_b, y_a - y_c) \iso \GMA{B}$, from which it easily follows that $\Ker \pi = (y_b, y_a - y_c)$.  To see that $(y_b, y_a - y_c) = (0 :_{\GMA{T}} y_b)$, we note that $M(T)$ is C-chordal since $T$ is a chordal graph, and $\{a, b, c\}$ is the only triangle of $T$ containing $b$.  Hence, the equality follows from part (b) of Proposition \ref{GMAs:and:contraction}.

(b)  Since $\pi:\GMA{T} \to \GMA{T}/(0: y_b) \iso \GMA{B}$ is a retract, it follows from \cite[Theorem 1]{algebra:retracts:and:Poincare:series} that $\pi$ is a large homomorphism in the sense of \cite[1.1]{large:homomorphisms}.  It then follows from \cite[2.1]{large:homomorphisms} that $\GMA{T} \to \GMA{T}/(y_b)$ is also large.  Hence, we have equalities of Poincar\'e series:
\[ \P^{\GMA{T}}(s, t) = \P^{\GMA{T}}_{\GMA{T}/(0: y_b)}(s, t)\P^{\GMA{B}}(s, t) \qquad \text{and} \qquad \P^{\GMA{T}}(s, t) = \P^{\GMA{T}}_{\GMA{T}/(y_b)}(s, t)\P^{\GMA{T}/(y_b)}(s, t). \]
From \eqref{defining:ideal:of:trampoline}, we see that $(Q_T, y_b) = Q_BS_T + (y_b, y_c^2, y_ay_c)$ so that 
\[ \GMA{T}/(y_b) \iso \GMA{B}[y_c]/(y_ay_c, y_c^2) \iso \GMA{B} \ltimes \GMA{B}/(y_a), \]
where $\GMA{B} \ltimes \GMA{B}/(y_a)$ denotes the Nagata idealization of $\GMA{B}/(y_a)$ as a module over $\GMA{B}$. (For example, see the proof of \cite[3.3]{QGKP1} for the last isomorphism.)  The idealization $\GMA{B} \ltimes \GMA{B}/(y_a)$ also has $\GMA{B}$ as a retract, and by \cite[Theorem 2]{Gulliksen} we have
\[ \P^{\GMA{B} \ltimes \GMA{B}/(y_a)}(s, t) = \frac{\P^{\GMA{B}}(s, t)}{1-st\P^{\GMA{B}}_{\GMA{B}/(y_a)}(s, t)}.\]
Finally, from the exact sequence $$0 \to \GMA{T}/(0 : y_b)(-1) \stackrel{y_b}{\to} \GMA{T} \to \GMA{T}/(y_b) \to 0\,$$ we have $$\P^{\GMA{T}}_{\GMA{T}/(y_b)}(s, t) = 1 + st\P^{\GMA{T}}_{\GMA{T}/(0: y_b)}(s, t).$$  Combining all of the preceding equalities yields
\begin{align*}
\P^{\GMA{T}}(s, t) &= 
\left(1 + st\frac{\P^{\GMA{T}}(s, t)}{\P^{\GMA{B}}(s, t)}\right)\left( \frac{\P^{\GMA{B}}(s, t)}{1-st\P^{\GMA{B}}_{\GMA{B}/(y_a)}(s, t)} \right) \\[1ex]
&= \frac{\P^{\GMA{B}}(s, t)}{1-st\P^{\GMA{B}}_{\GMA{B}/(y_a)}(s, t)} + \frac{st\P^{\GMA{T}}(s, t)}{1-st\P^{\GMA{B}}_{\GMA{B}/(y_a)}(s, t)}
\end{align*}
from which it follows that
\[ 
\P^{\GMA{T}}(s, t) =  \frac{\P^{\GMA{B}}(s, t)}{1-st(1 + P^{\GMA{B}}_{\GMA{B}/(y_a)}(s, t))}.
\qedhere \]
\end{proof}

\begin{proof}[Proof of Theorem \ref{Koszul:imples:strongly:chordal}]
We prove the contrapositive, using the forbidden minor characterization of strongly chordal graphs to show that if $G$ is not strongly chordal, then $\GMA{G}$ is not Koszul. If $G$ is not chordal, then $\GMA{G}$ is not even quadratic by Theorem \ref{quadratic:GMA:iff:chordal}, so we may assume that $G$ is chordal but not strongly chordal.  In that case, Theorem \ref{strongly:chordal:graphs} implies that $G$ has an induced $n$-trampoline $T$ for some $n \geq 3$.  We claim that $\GMA{T}$ is an algebra retract of $\GMA{G}$ so that there is an equality of Poincar\'e series $$\P^{\GMA{G}}(s, t) = \P^{\GMA{G}}_{\GMA{T}}(s, t)\P^{\GMA{T}}(s, t).$$  To see this, note that $\GMA{G}$ is quadratic since $G$ is chordal, and every 3-cycle of $G$ not contained in $T$ must have at least two edges not contained in $T$ as $T$ is an induced subgraph of $G$. Writing $\GMA{G} = S/Q_G$ for $S = \kk[y_i \mid i \in E(G)]$, we see that
\begin{align*} 
Q_G &= Q_TS + \Big(y_i(y_j - y_k), y_j(y_i - y_k) \mid \{i, j, k\} \;\text{a triangle of $G$, $i, j \notin E(T)$}\Big) 
 + (y_i^2 \mid i \notin E(T)),
\end{align*}
and so, it follows that the natural injection $\GMA{T} \to \GMA{G}$ admits a retract $\GMA{G} \to \GMA{T}$ by sending $y_i \mapsto 0$ for all $i \notin E(T)$.  By the above equality of Poincar\'e series, it suffices to show that $\GMA{T}$ is not Koszul to prove that $\GMA{G}$ is not Koszul.

We show that $\GMA{T}$ is not Koszul by induction on $n \geq 3$. Let $\GMA{B}$ and $a$ be as in the preceding lemma.  By part (b) of the lemma, we know that $\GMA{T}$ is Koszul if and only if $y_a\GMA{B}$ has a linear free resolution over $\GMA{B}$.  Since the broken $n$-trampoline $B$ is chordal, we know that its cycle matroid is C-chordal, and so, Proposition \ref{GMAs:and:contraction} yields that $$(0 :_R y_a) = (y_a)+ (y_i - y_j \mid \{a, i, j\} \;\text{a cycle of}\; B).$$ Hence, $\GMA{T}$ being Koszul is further equivalent to $(0 :_{\GMA{B}} y_a)$ having a linear free resolution over $\GMA{B}$.  We will prove this is not the case.

Suppose that $n = 3$, and let $B$ be the broken $3$-trampoline with edges labeled as shown below.
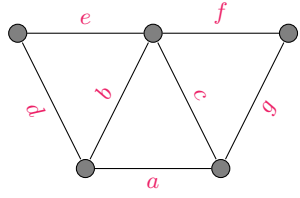
\vglue .2cm
\begin{center}
\edgelabeledgraph{
    -1/0/v_1//-120,
    0/2/v_2//90,
    1/0/v_3//-60,
    -2/2/w_1//120,
    2/2/w_2//60
    }{
    v_1/v_2/0.5/above/b, 
    v_1/v_3/0.5/below/a, 
    v_2/v_3/0.5/above/c,
    v_1/w_1/0.5/below/d,
    w_1/v_2/0.5/above/e,
    w_2/v_2/0.5/above/f,
    w_2/v_3/0.5/below/g}
\end{center}
Abusing notation slightly, we write $\GMA{B} = S/Q_B$, where $S = \kk[a,b,\dots, g]$.  We must show that $(0:_{\GMA{B}} a) = (a, b - c)$ does not have a linear resolution over $\GMA{B}$.  Let $R'' = \kk[b,c,,\dots, g]/I$, where 
\[
I = (d^2+2be, bd, de, b^2, e^2, g^2 + 2cf, cg, fg, c^2, f^2).
\]
After a change of coordinates sending 
$$a \mapsto a + b + c,\ \  d \mapsto b + d + e, \ \ g \mapsto c + f + g,$$ we see that 
\[ 
\GMA{B} \iso R' := \frac{S}{(a^2 + 2bc, ab, ac) + I} \iso \frac{R''[a]}{(a^2 + 2bc, ab, ac)}.
\]
Thus, it suffices to show that $(a + b + c, b - c)$ does not have a linear resolution over $R'$.  In fact, we will show that $((a+ b + c) :_{R'} b - c)$ has a minimal quadratic generator so that $(a + b + c, b - c)$ cannot have a linear presentation over $R'$. 

Since the ring $A = R''[a]/(a^2 + 2bc)$ is a free $R''$-module with basis $1, a$ and we have $a^2b = -2b^2c = 0$ and $a^2c = -2bc^2 = 0$ in $A$, it is easily checked that $(ab, ac)A = R''ab + R''ac$, from which it follows that $R' \iso R'' \oplus R''/(b, c)$ as an $R''$-module. We claim that 
\[ ((a + b + c) :_{R'} b - c) = ((bc) :_{R''} b - c)R' +  (a).\]
That the latter ideal is contained in the former follows from the facts that $a(b - c) = 0$ and $bc = (a + b + c)b$ in $R'$.  On the other hand, if $h \in ((a + b + c):_{R'} b - c)$ and we write $h = h_0 + ah_1$ for some $h_i \in R''$, then $(b - c)h = (b - c)h_0 \in (a + b + c)R'$ so that $$(b - c)h_0 = (a + b + c)(u_0 + au_1) = (b + c)u_0 -2bcu_1 + au_0.$$  By the above observation about the $R''$-module structure of $R'$, it follows that $u_0 \in (b, c)R''$ and that $(b - c)h_0 = (b + c)u_0 -2bcu_1 \in (bc)R''$ since $b^2 = c^2 = 0$, which establishes the claim. Hence, it further suffices to show that $((bc):_{R''} b - c)$ has a minimal quadratic generator.

We note that $R'' \iso P \tensor P$, where $$P = \kk[b, d, e]/(d^2 +2be, bd, de, b^2, e^2)$$ is easily seen to have a $\kk$-basis consisting of $1, b, d, e, be$. Let 
$$\ell = \alpha_1b + \alpha_2d + \alpha_3e + \alpha_4c + \alpha_5f + \alpha_6g$$ be a linear form in $R''$.  Then
\begin{align*} 
\ell b &= \alpha_3be + \alpha_4bc + \alpha_5bf + \alpha_6bg \\
\ell c &= \alpha_1bc + \alpha_2cd + \alpha_3ce + \alpha_5cf
\end{align*}
expresses the preceding products as linear combinations of distinct basis elements for $R''$.  As a consequence, we see that $((bc) :_{R''} b - c)_1$ is spanned by $b, c$, and $dg \in (0 :_{R''} b - c)$ is minimal quadratic generator for $((bc) :_{R''} b - c)$ as it is a basis element of $R''$ distinct from the basis elements appearing in the above products.  Thus, the graded M\"obius algebra of the 3-trampoline is not Koszul. 

Assume now that $n \geq 4$ and that the graded M\"obius algebra $\GMA{T'}$ of the $(n - 1)$-trampoline $T'$ is not Koszul.  Let $T$ and $B$ be the $n$-trampoline and the broken $n$-trampoline respectively.  We set $e_{i,j} = v_iv_j$, $b_k = v_kw_k$, and $c_k = v_{k+1}w_k$ in $B$, and abusing notation, we write $\GMA{B} = S/Q_B$, where 
\[ S = \kk[e_{i,j}, b_k, c_k \mid 1 \leq i < j \leq n, 1 \leq k \leq n-1]. \]  
As already noted above, it suffices to show that the ideal $(0 :_{\GMA{B}} e_{1,n})$ does not have linear free resolution over $\GMA{B}$ to show that $\GMA{T}$ is not Koszul.  If $(0 :_{\GMA{B}} e_{1,n})$ did have a linear free resolution over $\GMA{B}$, then since $\GMA{B}$ is Koszul by Corollary \ref{broken:trampolines:are:G-quadratic} it would follow from \cite[Theorem 2]{Koszul:algebras:and:regularity} that $\GMA{B}/(0 :_{\GMA{B}} e_{1,n})$ is also Koszul. However, by Proposition \ref{GMAs:and:contraction} and \cite[3.2.1]{Oxley}, we know that 
\[ \GMA{B}/(0 :_{\GMA{B}} e_{1,n}) \iso R_{\si(M(B))/e_{1,n}} = R_{\si(M(B/e_{1,n}))}, \] 
and since $\si(M(B/e_{1,n}))$ is isomorphic to the cycle matroid of the graph obtained by removing all loops and identifying all parallel edges of $B/e_{1,n}$, we see that $\GMA{B}/(0 :_{\GMA{B}} e_{1,n}) \iso \GMA{T'}$.  Hence, by induction, we have that $\GMA{T}$ is not Koszul.
\end{proof}

We summarize the more refined  characterizations of Koszulness for graded M\"obius and Orlik-Solomon algebras of graphic matroids in the Figure \ref{graphic:quadracity:implications} below.  By Figure~\ref{graphic:quadracity:implications}, if $M$ is a graphic matroid that is strongly T-chordal, then $M$ is supersolvable. 

\begin{figure}[htb]
    \begin{center}
    \begin{tikzcd}[column sep = 3em, row sep = 3.5 em]
    \substack{\text{\normalsize $\OS{G}$} \\ \text{\normalsize quadratic GB}}
    \dar[Leftrightarrow]
    &
    \substack{\text{\normalsize $M(G)$} \\ \text{\normalsize supersolvable}}
    \lar[Leftrightarrow, swap]
    &
    \substack{\text{\normalsize $\GMA{G}$} \\ \text{\normalsize quadratic GB}} \dar[Leftrightarrow]
    &
    \substack{\text{\normalsize $M(G)$} \\ \text{\normalsize strongly T-chordal}}
    \lar[Leftrightarrow]
    \\
    \OS{G} \,\text{Koszul}
    \dar[Leftrightarrow]
    &
    &
    \GMA{G} \, \text{Koszul} \dar[Rightarrow] \arrow[Leftrightarrow]{r}
    &
    G \; \text{strongly chordal}
    \arrow[Leftrightarrow]{u}
    \\
    \OS{G} \,\text{quadratic}
    &
    G \; \text{chordal}  
    \arrow[Leftrightarrow]{l}
    \arrow[Leftrightarrow]{uu}
    \rar[Leftrightarrow, swap]
    &
    \GMA{G} \, \text{quadratic} 
    \end{tikzcd}
    \end{center}
    \caption{Chordality and Quadracity Properties for Graded M\"obius and Orlik-Solomon Algebras of Graphic Matroids}
    \label{graphic:quadracity:implications}
\end{figure}
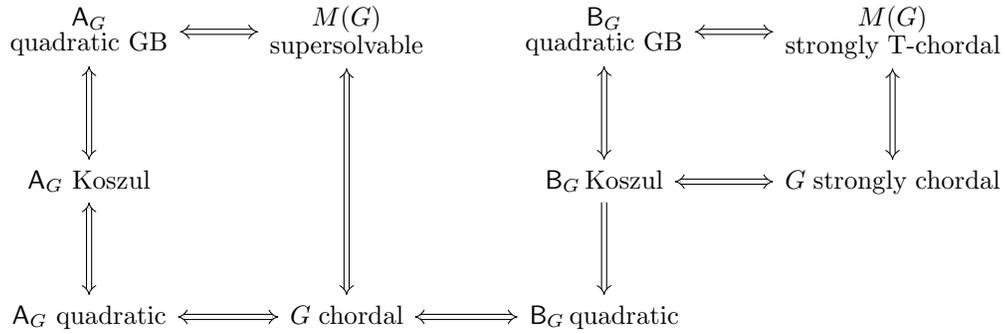

 \begin{prob}
 If $M$ is an arbitrary matroid that is strongly T-chordal, is $M$ supersolvable?
 \end{prob}

For both Orlik-Solomon algebras and graded M\"obius algebras of graphic matroids, the Koszul property is equivalent to the defining ideal having a quadratic Gr\"obner basis in the standard presentation.  This equivalence for Orlik Solomon algebras of arbitrary arrangements or matroids is an open question.  The following example shows that these notions are distinct for graded M\"obius algebras of arbitrary matroids.

\begin{example}\label{ag23}
    If $M$ is a graphic matroid, Figure~\ref{graphic:quadracity:implications} shows that $\OS{M}$ is Koszul whenever $\GMA{M}$ is Koszul, and the  converse fails for chordal graphs that are not strongly chordal.  Hence, $\GMA{M}$ can be viewed as a finer invariant of $M$ than $\OS{M}$ when $M$ is graphic.  However, this is not the case for arbitrary matroids.  
    
    Let $AG(2,3)$ be the matroid associated to the affine plane over $\ZZ/3\ZZ$.  This is a rank $3$ matroid on $9$ elements whose Orlik-Solomon algebra is quadratic but not Koszul; see \cite[Example 4.5]{Irena:hyperplane:arrangements}.  However, a filtration argument shows that $\GMA{AG(2,3)}$ is Koszul.  On the other hand, it is easy to check that the defining ideal of $\GMA{AG(2,3)}$ does not have a quadratic Gr\"obner basis with respect to any monomial order, meaning that not all Koszul graded M\"obius algebras have quadratic Gr\"obner bases in the natural presentation.  We note that $AG(2,3)$ is C-chordal, but an exhaustive computer search shows that it is not strongly T-chordal.  Thus C-chordality does not imply strong T-chordality in general.
\end{example}

\section{Linearity for finitely many steps} \label{open:problems}

\noindent
We conclude the paper by raising an open problem.  Roos \cite{Roos93} constructed a family of quadratic $\kk$-algebras $\{A_n\}$ indexed over the natural numbers $n \geq 2$ for which the resolution of $\kk$ over $A_n$ is linear for exactly $n$ steps.  Thus, one cannot check the Koszul property definitively by computing the first few (tens or even billions of) steps in the resolution of the ground field.  While all of the rings $A_n$ are Artinian with identical Hilbert functions, the proof of the behavior of the resolution is intricate.  Computations suggest that graded M\"obius algebras may provide a  natural family of quadratic algebras with even more extreme homological behavior than Roos' family.

Specifically, let $R_n$ be the graded M\"obius algebra of an $n$-trampoline for any $n \geq 3$.  It follows from Theorem~\ref{Koszul:imples:strongly:chordal} that $R_n$ is not Koszul for every $n$. Computation with Macaulay2 shows that for $n \leq 6$, the Betti table of $\kk$ over $R_n$ has the following form:\vspace{2ex}

\begin{center}
\begin{tabular}{|c|c|}
\hline
$n = 3$ & $n = 4$ \\
\hline
    \begin{minipage}{0.45\textwidth}
    \vspace{2 ex}
    \begin{center}
    \renewcommand{\arraystretch}{1.1}
\begin{tabular}{c|ccccc}
        & 0 & 1 & 2 & 3 & 4\\
       \hline
       0 & 1 & 9 & 53 & 260 & 1,156 \\
       1 &--&--&--&--&1
       \end{tabular}
               \end{center}
        \vspace{2 ex}
       \end{minipage}
&
    \begin{minipage}{0.45\textwidth}
    \vspace{2 ex}
    \begin{center}
    \renewcommand{\arraystretch}{1.1}
    \begin{tabular}{c|cccccc}
        & 0 & 1 & 2 & 3 & 4&5\\
       \hline
       0 & 1 & 14 & 121 & 841 & 5,191 & 29,886 \\
       1 &--&--&--&--&--&--\\
       2 &--&--&--&--&--&1
       \end{tabular}
                      \end{center}
        \vspace{2 ex}
       \end{minipage}      
       \\
       \hline 
\end{tabular}
\\[2ex]

\begin{tabular}{|c|c|}
\hline
$n = 5$ & $n = 6$ \\
\hline
    \begin{minipage}{0.45\textwidth}
    \vspace{2 ex}
    \begin{center}
    \renewcommand{\arraystretch}{1.1}
    \begin{tabular}{c|ccccccc}
          & 0 & 1 & 2 & 3 & 4 & 5 & 6 \\
       \hline
       0 & 1 & 20 & $*$ & $*$ & $*$ & $*$ & 1,083,885 \\
       1 &--&--&--&--&--&--&--\\
       2 &--&--&--&--&--&--&--\\
       3 &--&--&--&--&--&--&1
       \end{tabular}
                             \end{center}
        \vspace{2 ex}
       \end{minipage}      
&
    \begin{minipage}{0.45\textwidth}
    \vspace{2 ex}
    \begin{center}
    \renewcommand{\arraystretch}{1.1}
    \begin{tabular}{c|cccccccc}
      & 0 & 1 & 2 & 3 & 4 & 5 & 6 & 7 \\
       \hline
       0 & 1 & 27 & $*$ & $*$ & $*$ & $*$ & $*$ & 51,581,417 \\
       1 &--&--&--&--&--&--&--&--\\
       2 &--&--&--&--&--&--&--&--\\
       3 &--&--&--&--&--&--&--&--\\
       4 &--&--&--&--&--&--&--&1
       \end{tabular}
                             \end{center}
        \vspace{2 ex}
       \end{minipage}      
       \\
\hline
\end{tabular}
\end{center}
\vspace{2ex}
Furthermore, unlike Roos's examples which rely heavily on characteristic zero methods, the Betti tables of $\kk$ over $R_n$ seem to be independent of the characteristic of the base field. It is natural to ask if  these patterns hold for  all $n \geq 3$:

\begin{prob}
Consider $R_n$ over a ground field $\kk$.
Do the following properties hold for every  $n \geq 3$?
\begin{enumerate}[label = \textnormal{(\alph*)}]
\item The minimal free  resolution of $\kk$ over $R_n$ is linear for exactly $n$ steps, that is, the first non-linear Betti number appears in homological degree $n+1$.
\item The homologically-first nonlinear syzygy  of $\kk$ over $R_n$ has degree $(n+1)+(n-2)=2n-1$, so in the Betti table it appears in the $(n-2)$-th strand
(whereas the first nonlinear syzygy for Roos' family always appears in the first strand).
\item
The homologically-first non-linear Betti number of 
 $\kk$ over $R_n$ is equal to 1. 
\end{enumerate}

\end{prob}

\section*{Acknowledgements}

The authors thank Sophie Spirkl for very helpful conversations and thank Michael DiPasquale for pointing the authors to the paper by Tran and Tsujie.  Computations with Macaulay2 \cite{M2} were very useful while working on this project.  Finally, the authors thank the anonymous referee for comments that helped improve the exposition in this paper.  LaClair was partially supported by National Science Foundation grant DMS--2100288 and by Simons Foundation Collaboration grant for Mathematicians \#580839.  Mastroeni was partially supported by an AMS-Simons Travel Grant.  McCullough was partially supported by National Science Foundation grants DMS--1900792 and  DMS--2401256.  Peeva was partially supported by National Science Foundation grants DMS--2001064 and DMS-2401238.

\end{spacing}

\end{document}